\documentclass[a4paper,11pt]{amsart}          			%Use the amsart class 
\usepackage[twoside=false]{geometry}  %Geometry
\usepackage{amsfonts}                                   %Adds, among other things, blackboard bold font
\usepackage{amsmath}                                    %Math environments
\usepackage[initials]{amsrefs}                          %References
\usepackage{amssymb}                                    %Symbols
\usepackage{amsthm}                                     %Theorem environments
\usepackage[english]{babel}                             %Localization
\usepackage{caption}                                    %Memoir breaks things like footnotes in caption
\usepackage{cite}                                       %Makes up for some of the shortcomings of losing amsrefs
\usepackage[shortlabels,inline]{enumitem}               %Customization of enumerate
\usepackage{esint}                                      %Adds several integral symbols
\usepackage[T1]{fontenc}                                %Adds more letters
\usepackage{graphicx}                                   %Graphics
\usepackage[utf8]{inputenc}                             %UTF 8
\usepackage{lipsum}                                     %Lorem ipsum generation
\usepackage{lmodern}                                    %Latin Modern font
\usepackage{mathrsfs}                                   %Script fonts
\usepackage{mathtools}                                  %Various symbols (like \coloneqq) and math stuff
\usepackage{microtype}                                  %Makes documents prettier and more readable.
\usepackage{multicol}
\usepackage{multirow}                                   %Adds the ability to make cells span multiple rows.
\usepackage[utf8]{inputenc}                             %UTF8 text
\usepackage{pgf}                                        %For tikz
\usepackage{pgfplots}                                   %Plots
\usepackage{siunitx}                                    %Scientific units
\usepackage{subcaption}                                 %Allows captions for subfigures
\usepackage{tikz}                                       %Drawing
\usepackage{tikz-cd}                                    %Commutative diagrams
\usepackage{tikz-3dplot}                                %Allows for basic 3d in tikz
\usepackage{xcolor}                                     %Colors!

\usepackage[hidelinks,hypertexnames=false]{hyperref}    %Links to reference. Needs to be before cleveref.
\usepackage[nameinlink]{cleveref}                       %Automatic handling of reference type

\usetikzlibrary{hobby, patterns, fadings,decorations,positioning,calc}
\usepgfplotslibrary{fillbetween,patchplots}
\pgfplotsset{compat=1.14}

\newcommand{\map}[3]{ #1 \colon #2 \to #3 }
\newcommand{\wildcard}{\makebox[\widthof{$x$}][c]{$\cdot$}}

\newcommand{\mc}[1]{{\mathcal{#1}}}

\newcommand{\R}{ \mathbb{R} }

\newcommand{\N}{ \mathbb{N} }

\DeclarePairedDelimiter{\parn}{\lparen}{\rparen}     % ( )
\DeclarePairedDelimiter{\brac}{\lbrace}{\rbrace}    % { }
\DeclarePairedDelimiter{\brak}{\lbrack}{\rbrack}    % [ ]
\DeclarePairedDelimiter{\angl}{\langle}{\rangle}    % < >

\DeclarePairedDelimiter{\abs}{\lvert}{\rvert}       % | |
\DeclarePairedDelimiter{\norm}{\lVert}{\rVert}      % || ||

\DeclareMathOperator{\img}{im}
\DeclareMathOperator{\ind}{ind}

\DeclareMathOperator{\sgn}{sgn}

\DeclareMathOperator{\lin}{Lin}
\DeclareMathOperator{\lspan}{span}

\theoremstyle{plain}
\newtheorem{theorem}{Theorem}

\newtheorem{lemma}[theorem]{Lemma}
\newtheorem{proposition}[theorem]{Proposition}

\theoremstyle{definition}

\theoremstyle{remark}
\newtheorem{remark}[theorem]{Remark}
\newtheorem{assumption}[theorem]{Assumption}

\crefname{assumption}{assumption}{assumptions}   %Missing plural
\crefname{enumi}{part}{parts}   %Change item -> part in cleveref

\title[Global bifurcation of waves with critical layers]{Global bifurcation of waves with multiple critical layers}
\author{Kristoffer Varholm}
\address{Department of Mathematical Sciences, Norwegian University of Science and Technology, 7491 Trondheim, Norway}
\email{kristoffer.varholm@ntnu.no}
\thanks{The author acknowledges the support of the projects \emph{Nonlinear Water Waves} (grant no. 231668) and \emph{Waves and Nonlinear Phenomena} (grant no. 250070), from the Research Council of Norway. They also wish to thank Gabriele Brüll, Mats Ehrnström, Eugen Varvaruca, Erik Wahlén, and Miles Wheeler for fruitful conversations concerning aspects of the work.}

\begin{document}
\begin{abstract}
    Analytic global bifurcation theory is used to construct a large variety of families of steady periodic two-dimensional gravity water waves with real-analytic vorticity distributions, propagating in an incompressible fluid. The waves that are constructed can possess an arbitrary number of interior stagnation points in the fluid, and corresponding critical layers consisting of closed streamlines. This is made possible by the use of the so-called naive flattening transform, which has previously only been used for local bifurcation.
\end{abstract}
\maketitle
\section{Introduction}
    In this paper, our concern shall be two-dimensional traveling water waves, propagating in an inviscid and incompressible fluid of finite depth atop a flat bed. The waves will be purely gravitational --- that is, we neglect the influence of surface tension. Moreover, we make the assumption that the waves have not overturned, meaning that the free surface can be described as the graph of a function, which we shall call $\map{\eta}{\R}{\R}$ in the sequel.

    The steady-frame fluid domain, stationary with respect to the wave, will be denoted by
    \[
        \Omega_\eta \coloneqq \brac*{ (x,y) \in \R^2 : 0 < y < d + \eta(x)},
    \]
    where $d > 0$ represents the unperturbed fluid depth; with $x$ marking the horizontal direction, and $y$ the vertical direction. Furthermore, we will write
    \[
        S_\eta \coloneqq \brac*{(x,d + \eta(x)) : x \in \R}
    \]
    for the free surface, and
    \[
        B \coloneqq \{(x,0) : x \in \R\}
    \]
    to signify the flat bed. These components of $\partial \Omega_\eta$ are assumed to be positively separated, with $S_\eta$ situated above $B$.

    The waves are required to satisfy the steady incompressible Euler equations
    \begin{subequations}
        \label{eq:steady_water_wave_problem}
        \begin{equation}
            \label{eq:incompressible_euler}
            \begin{gathered}
                \nabla \cdot \brak*{\parn*{u-ce_x} \otimes \parn*{u-ce_x}} + \nabla(p+gy)=0,\\
                \nabla \cdot u = 0,
            \end{gathered}
        \end{equation}
        in $\Omega_\eta$, where $\map{u}{\Omega_\eta}{\R^2}$ is the velocity field, and $\map{p}{\Omega_\eta}{\R}$ is the pressure. The constant $c > 0$ is the wave speed,  and $g > 0$ is known as the acceleration due to gravity, while $e_x \coloneqq (1,0)$ is the horizontal unit vector. We may interpret the individual equations in \eqref{eq:incompressible_euler} as representing conservation of momentum and mass, respectively.

        To finish the description of the governing equations for steady water waves, we also require boundary conditions: First, we have the \emph{kinematic} boundary conditions, which read
        \begin{equation}
            \label{eq:kinematic_surface}
            \eta^\perp \cdot (u-ce_x) = 0 \qquad \text{on $S_\eta$},
        \end{equation}
        and
        \begin{equation}
            \label{eq:kinematic_bed}
            e_y \cdot u = 0 \qquad \text{on $B$},
        \end{equation}
        naturally ``attaching'' the boundary of the fluid domain to the velocity field. Here, $\eta^\perp \coloneqq (-\partial_x \eta,1)$ yields the non-normalized normal vector on $S_\eta$ in terms of the surface profile $\eta$.

        The final boundary condition is the \emph{dynamic} boundary condition
        \begin{equation}
            \label{eq:dynamic}
            p = 0 \qquad \text{on $S_\eta$},
        \end{equation}
        ensuring that the pressure is continuous across the interface. This is where surface tension would have entered, had we not neglected it. Collectively, \eqref{eq:incompressible_euler}--\eqref{eq:dynamic} is known as the \emph{steady water-wave problem}.
    \end{subequations}

    Of particular interest to us are \emph{rotational} steady waves, for which the scalar circulation density
    \[
        \omega \coloneqq \nabla^\perp \cdot u, \qquad \text{where }\nabla^\perp \coloneqq (-\partial_y,\partial_x),
    \]
    does not vanish identically. This quantity is known as the \emph{vorticity} of the fluid. Conveniently, if a \emph{stream function} $\map{\psi}{\Omega_\eta}{\R}$ is introduced through $\nabla^\perp \psi \coloneqq u-ce_x$, then the identity
    \[
        \omega = \nabla^\perp \cdot \nabla^\perp \psi = \Delta \psi
    \]
    holds throughout $\Omega_\eta$.

    We can go further than this still: If we have a smooth solution of the steady water-wave problem \eqref{eq:steady_water_wave_problem}, and $c - u \cdot e_x =\partial_y \psi > 0$ in $\Omega_\eta$, then there exists a \emph{vorticity distribution} $\map{\gamma}{\R}{\R}$ such that
    \begin{subequations}
        \label{eq:water_wave_problem_stream}
        \begin{equation}
            \label{eq:helmholtz}
            \Delta \psi + \gamma(\psi) = 0 \qquad \text{in $\Omega_\eta$},
        \end{equation}
        see e.g. \cite[Section 2]{Constantin04Exact}. Of course, this merely constitutes a sufficient condition for $\gamma$ to exist. There is nothing preventing us from postulating the existence of a vorticity distribution, even when the hypothesis above fails in the presence of interior \emph{stagnation points} where $u-ce_x = \nabla^\perp \psi = 0$. Having access to a vorticity distribution is highly convenient mathematically.

        By employing elementary vector-calculus identities, we find that
        \[
            \nabla \cdot \parn*{\nabla^\perp \psi \otimes \nabla^\perp \psi} = \nabla \parn*{\frac{1}{2}\abs*{\nabla \psi}^2 + \Gamma(\psi)}, \qquad \Gamma(t) \coloneqq \int_0^t \gamma(s)\,ds,
        \]
        for solutions of \eqref{eq:helmholtz}. Combining this identity with \eqref{eq:incompressible_euler} and \eqref{eq:dynamic}, we are led to the free-surface \emph{Bernoulli equation}
        \begin{equation}
            \label{eq:free_surface_bernoulli_pre_scaling}
            \frac{1}{2} \abs*{\nabla \psi}^2 +g\eta = Q\qquad \text{on $S_\eta$},
        \end{equation}
        for some constant $Q > - gd$. Together with \eqref{eq:helmholtz} and the demand that
        \begin{align}
            \psi & = \mu \qquad \text{on $S_\eta$}, \label{eq:psi_surface} \\
            \psi & = \Upsilon \qquad \text{on $B$}, \label{eq:psi_bottom}
        \end{align}
        for two constants $\mu,\Upsilon \in \R$, which is the form that the kinematic boundary conditions \eqref{eq:kinematic_surface} and \eqref{eq:kinematic_bed} take for stream functions, these equations give rise to solutions of the steady water-wave problem in \eqref{eq:steady_water_wave_problem}.

    \end{subequations}

    If we appeal to the integral
    \[
        \Upsilon - \mu = \int_0^{d+\eta(x)} (u(x,y)\cdot e_x - c)\,dy,
    \]
    which is independent of the choice of $x \in \R$, we see that the difference between the constants in \eqref{eq:psi_surface} and \eqref{eq:psi_bottom} may be interpreted as a relative mass flux.

    \subsection{Previous work}
        The mathematical study of solutions to variants of the steady water-wave problem \eqref{eq:steady_water_wave_problem} has a rich and extensive history --- going back hundreds of years. Most of the earlier literature concerned irrotational waves, where the reader may find surveys such as \cite{Toland96Stokes,Groves04Steady} of interest. Comparatively, the study of \emph{rotational} steady waves specifically is much more recent, especially when they are allowed to stagnate. Rotation is of course ubiquitous in nature, essentially being induced whenever there are non-conservative forces at play, but also introduces new non-trivial mathematical challenges.

        One could argue that the first rotational water-wave result was the explicit infinite-depth Gerstner wave \cite{Gerstner09Theorie}; see also \cite{Constantin11Nonlinear} for a modern treatment, including recent developments in the field of nonlinear water waves more generally. However, a much more compelling case can be made for the doctoral thesis \cite{Dubreil-Jacotin34Sur}. There, Dubreil-Jacotin introduced the semi-hodograph transform for \eqref{eq:water_wave_problem_stream}, treating the stream function as the vertical variable. They subsequently used the transform in an existence theorem for small-amplitude periodic solutions, and it has since seen wide use and become a staple tool in the field. In particular, we must single out its use in the seminal paper \cite{Constantin04Exact}, which was the first large-amplitude result in the same setting.

        One significant downside of the semi-hodograph transform is that it precludes the presence of interior stagnation points, or their corresponding critical layers of closed streamlines, in the fluid. The reason for this is that $\psi_y$ must necessarily have a definite sign in order to enable the use of $\psi$ as a vertical variable. Therefore, if stagnation is a desired feature, a different way of dealing with the free boundary must be utilized.

        An early existence result for stagnant waves was \cite{Ehrnstroem08Linear}, furnishing \emph{linear} stagnant waves with constant vorticity. This paper would lead to the first nonlinear existence result in \cite{Wahlen09Steady}, where Wahlén constructed small-amplitude waves with one critical layer for constant $\gamma$. Instead of the semi-hodograph transform, they used the same naive flattening transform that we shall soon introduce (see \eqref{def:flattening_transform}). In the decade following this paper, there has been a flurry of activity concerning waves with stagnation points: In \cite{Constantin11Steady} a different approach than \cite{Wahlen09Steady} was used, restating the problem as a pseudodifferential equation using conformal mapping. While highly specialized for constant vorticity, the framework is elegant and potentially allows for overhanging waves. The authors of \cite{Constantin11Steady} would later go on to further develop this framework with Strauss in \cite{Constantin16Global}, establishing the existence of \emph{large} amplitude waves. We should also mention here that there is another global result in the presence of capillary effects \cite{Matioc14Global}.

        At the same time, there has been a parallel endeavor of considering more ``interesting'' vorticity distributions, expanding on the use of the flattening transform from \cite{Wahlen09Steady}. Even taking the step up to affine vorticity distributions \cite{Ehrnstroem11Steady,Ehrnstroem12Steady,Ehrnstroem15Trimodal,Aasen18Traveling} admits waves with an arbitrary number of critical layers. These works also provide results for bimodal \cite{Ehrnstroem11Steady,Aasen18Traveling}, or even trimodal \cite{Ehrnstroem15Trimodal} waves. Small-amplitude solutions for very general vorticity distributions $\gamma$ were examined in \cite{Kozlov14Dispersion}. See also \cite{Kozlov18N} for a recent result involving $n$-modal waves for ``almost''-affine vorticity distributions, giving a partial answer to a question posed in \cite{Ehrnstroem15Trimodal}.

        In this paper, our goal is to, in a sense, unite these two efforts: We show that the framework of \cite{Wahlen09Steady,Ehrnstroem11Steady} can be extended in such a way that it can be used for large-amplitude waves as well.

        There are, of course, several papers on stagnant waves that do not fit neatly into our categorizations above. For instance, a recent paper \cite{Kozlov19Small} constructs small-amplitude \emph{non-symmetric} waves with critical layers using a spatial-dynamics approach, as opposed to the bifurcation-theoretic nature of the above results. There is also a preprint \cite{Kozlov19Solitary} on solitary waves with constant vorticity and a critical layer connecting with the bed, again using spatial dynamics. Finally, there are several papers on solitary capillary-gravity waves with compactly supported vorticity \cite{Shatah13Travelling,Varholm16Solitary}, including two recent ones dealing with the stability of such waves \cite{Varholm20Stability,Le19Existence}. These waves are stagnant, but we remark that the waves with immersed point vortices do not possess vorticity distributions in the traditional sense.

    \subsection{Plan for this article}

        In \Cref{section:formulation}, we formulate the problem, and describe its linearization. Of noteworthy importance is the key generalization of the so-called $\mc{T}$-isomorphism from \cite{Ehrnstroem11Steady} to non-trivial solutions. \Cref{section:kernel_local_bifurcation} is used to study the kernel of the linearization, in particular resulting in \Cref{thm:kernel}. This section is also used to state our local bifurcation result, \Cref{thm:local_bifurcation}. Following this, we briefly discuss the special cases of constant and affine vorticity in \Cref{sec:explicit}. Finally, \Cref{section:global_bifurcation} extends the local curves to \emph{global} ones, which is the central event of this paper. Our main result here is \Cref{thm:global_bifurcation}, which concerns the global solution curves obtained by applying analytic global bifurcation theory. An interesting feature of its proof is the use of an alternative near-surface flattening to prove the necessary compactness.

\section{Formulation}
    \label{section:formulation}
    After a convenient choice of scaling, we may set the unperturbed depth and gravitational acceleration to $d = g \equiv 1$, switching out \eqref{eq:free_surface_bernoulli_pre_scaling} for
    \begin{equation}
        \label{eq:free_surface_bernoulli}
        \frac{1}{2} \abs*{\nabla \psi}^2 +\eta = Q\qquad \text{on $S_\eta$}, \tag{\ref{eq:free_surface_bernoulli_pre_scaling}$\star$}
    \end{equation}
    instead. At this point, we make a regularity assumption, the first half of which is necessary for us to be able to apply \emph{analytic} global bifurcation theory later.

    \begin{assumption}[Regularity of $\gamma$]
        \label{ass:gamma_regularity}
        The vorticity distribution $\map{\gamma}{\R}{\R}$ is real analytic, with bounded derivative.
    \end{assumption}

    The trivial solutions of \eqref{eq:water_wave_problem_stream} are those corresponding to parallel flows beneath a flat surface, and in particular those for which $\eta \equiv 0$. We \emph{define} the trivial stream function $\overline{\psi} = \overline{\psi}(\Lambda)$ to be the unique solution --- the existence of which is ensured by \Cref{ass:gamma_regularity} --- of the initial value problem
    \begin{equation}
        \label{eq:trivial_solution_definition}
        \begin{gathered}
            \overline{\psi}'' + \gamma(\overline{\psi}) = 0, \quad \text {for $y \in (0,1)$} \\
            \overline{\psi}(1)= \mu, \qquad \overline{\psi}'(1) = \lambda,
        \end{gathered}
    \end{equation}
    for $\Lambda \coloneqq (\mu,\lambda)$ in the set
    \begin{equation}
        \label{def:permissible_parameters}
        \mc{U} \coloneqq \{\underbrace{(\mu,\lambda)}_{\Lambda} \in \R^2 : \lambda \neq 0\}
    \end{equation}
    of permissible parameters. Here, the restriction on $\lambda$ ensures that there is no technically problematic surface stagnation present at the trivial solution. The corresponding values of $Q$ and $\Upsilon$ are determined from \eqref{eq:free_surface_bernoulli} and \eqref{eq:psi_bottom}, namely
    \begin{align}
        Q(\Lambda)        & = \frac{1}{2}\lambda^2, \label{eq:Q_value} \\
        \Upsilon(\Lambda) & = \overline{\psi}(0;\Lambda). \notag
    \end{align}
    We will often leave out explicit dependence on $\Lambda$ from our notation for readability, especially for $\overline{\psi}$.

    \begin{remark}
        We mention that the solutions of \eqref{eq:trivial_solution_definition} can be written down explicitly only for very special choices of $\gamma$, such as when the vorticity distribution is either constant or affine.
    \end{remark}

    By flattening the fluid domain through what we shall call the \emph{naive flattening transform} $\map{\Pi}{\Omega_\eta}{\Omega_0}$, defined by
    \begin{equation}
        \label{def:flattening_transform}
        \Pi(x,y) = \parn*{x,\frac{y}{1+\eta(x)}},
    \end{equation}
    the water wave problem in \eqref{eq:water_wave_problem_stream} becomes
    \begin{equation}
        \label{eq:water_wave_problem_flattened}
        \begin{aligned}
            \parn*{\partial_x - \frac{s\eta_x}{1+\eta}\partial_s}^2 \hat{\psi} + \frac{1}{(1+\eta)^2}\hat{\psi}_{ss} + \gamma(\hat{\psi}) & =0         &  & \text{in $\Omega_0$,} \\
            \frac{1+\eta_x^2}{2(1+\eta)^2}\hat{\psi}_s^2 + \eta                                                                           & = Q        &  & \text{on $S_0$,}      \\
            \hat{\psi}                                                                                                                    & = \mu      &  & \text{on $S_0$,}      \\
            \hat{\psi}                                                                                                                    & = \Upsilon &  & \text{on $B$,}
        \end{aligned}
    \end{equation}
    where $s$ is used to distinguish the vertical variable in the flattened domain. For notational simplicity, we will henceforth use $\Omega \coloneqq \Omega_0$ and $S \coloneqq S_0$.

    Write now
    \begin{equation}
        \label{eq:def_hat_psi}
        \hat{\psi} = \hat{\psi}(\hat{\varphi},\Lambda) \coloneqq \overline{\psi}(\Lambda) + \hat{\varphi},
    \end{equation}
    where $\hat{\varphi}$ is a disturbance from $\overline{\psi}$ that vanishes at both the bottom and surface. The trivial solution in this definition takes care of the Dirichlet boundary conditions in \eqref{eq:water_wave_problem_flattened}. We will typically use the notation $w = (\eta,\hat{\varphi})$, for the pairs living in the space
    \begin{equation}
        \label{def:X}
        X = X_1 \times \hat{X}_2 \coloneqq C_{\kappa,\textnormal{e}}^{2,\beta}(\R) \times \{\hat{\varphi} \in C_{\kappa,\textnormal{e}}^{2,\beta}(\overline{\Omega}) : \hat{\varphi}|_S=\hat{\varphi}|_B=0\},
    \end{equation}
    for some fixed Hölder exponent $\beta \in (0,1)$. Here, the subscripts denote $2\pi/\kappa$-periodicity and evenness in the horizontal direction, respectively.

    If we further define the open subset
    \begin{equation}
        \label{eq:open_set}
        \mc{O} \coloneqq \{(w,\Lambda) \in X \times \mc{U} : 1+\eta > 0,\sgn(\lambda)\hat{\psi}_s|_S > 0\}
    \end{equation}
    of $X \times \R^2$, we may define the \emph{analytic} map $\map{\mc{F}=(\mc{F}_1,\mc{F}_2)}{\mc{O}}{Y}$ by
    \begin{equation}
        \label{eq:mcF_definition}
        \begin{aligned}
            \mc{F}_1(w,\Lambda) & \coloneqq \frac{1+\eta_x^2}{2(1+\eta)^2}\parn{\hat{\psi}_s|_S}^2 + \eta - \frac{1}{2}\lambda^2,                                        \\
            \mc{F}_2(w,\Lambda) & \coloneqq\parn*{\partial_x - \frac{s\eta_x}{1+\eta}\partial_s}^2\hat{\psi} + \frac{1}{(1+\eta)^2}\hat{\psi}_{ss} + \gamma(\hat{\psi}),
        \end{aligned}
    \end{equation}
    where the natural codomain of $\mc{F}$ is the space
    \[
        Y = Y_1 \times Y_2 \coloneqq C_{\kappa,\textnormal{e}}^{1,\beta}(\R) \times C_{\kappa,\textnormal{e}}^{\beta}(\overline{\Omega}).
    \]
    Note in particular that $\Lambda \mapsto \overline{\psi}(\Lambda)$ defines an analytic map from $\R^2$ into the space $C^{2,\beta}([0,1]) \eqqcolon V$, by a simple argument involving the implicit function theorem applied to the map $\map{F}{V \times \R^2}{V}$ defined through
    \[
        F(\zeta,\Lambda)(s) \coloneqq \zeta(s) - \mu - \lambda(s-1) + \int_1^s \int_1^t \gamma(\zeta(r))\,dr\,dt.
    \]

    The idea of the definition in \eqref{eq:mcF_definition} is that it combines \eqref{eq:water_wave_problem_flattened} with \eqref{eq:def_hat_psi}, and the value of of $Q$ from \eqref{eq:Q_value}. Thus $(0,\Lambda)$ is a solution of the equation
    \begin{equation}
        \label{eq:F_zero}
        \mc{F}(w,\Lambda) =0
    \end{equation}
    in $\mc{O}$ for every $\Lambda \in \mc{U}$, by construction. Moreover, these are the only solutions with the flat surface $\eta \equiv 0$.

    \begin{remark}
        \label{rem:open_set}
        The last condition for membership in $\mc{O}$ ensures that there is no stagnation on the surface, and that
        \[
            \sgn(\hat{\psi}_s|_S) = \sgn(\lambda),
        \]
        whereupon the line segment between $(0,\Lambda)$ and $(w,\Lambda)$ is always contained in $\mc{O}$ for any $(w,\Lambda) \in \mc{O}$. A further implication is that the slice
        \begin{equation}
            \label{def:open_set_slice}
            \mc{O}_\lambda \coloneqq \{(w,\mu) \in X \times \R : (w,\mu,\lambda) \in \mc{O}\}
        \end{equation}
        is connected for any fixed $\lambda \neq 0$.
    \end{remark}

    Our objective from here on is to further investigate the solution set of \eqref{eq:F_zero} in $\mc{O}$. The first observation we make is that solutions of \eqref{eq:F_zero} are more regular than generic elements of $\mc{O}$, as a consequence of the following theorem:

    \begin{theorem}[Analyticity of solutions]
        \label{thm:analyticity_of_solutions}
        Suppose that $(w,\Lambda) \in \mc{O}$ is a solution of \eqref{eq:F_zero} under \Cref{ass:gamma_regularity}. Then:
        \begin{enumerate}[(i)]
            \item The surface profile $\eta$ is analytic.
            \item The stream function $\hat{\psi}$ extends to an analytic function on an open set containing $\overline{\Omega}$.
        \end{enumerate}
    \end{theorem}
    \begin{proof}
        The proof is essentially the same as the one for \cite[Theorem 2.5]{Aasen18Traveling}, but using \emph{nonlinear} elliptic regularity theory instead of linear theory; see for instance~\cite{Morrey58Analyticity}.
    \end{proof}

    \subsection{Linearization around a solution}
        In preparation for both local and global bifurcation, we require the linearization of \eqref{eq:F_zero} around its solutions. It is straightforward, but laborious, to compute the partial derivatives
        \begin{align*}
            D_\eta \mc{F}_1 (w,\Lambda)H                     & =\parn*{1 -\frac{1+\eta_x^2}{(1+\eta)^3}\hat{\psi}_s^2}H + \frac{\eta_x\hat{\psi}_s^2}{(1+\eta)^2}H_x \\
            D_{\hat{\varphi}} \mc{F}_1 (w,\Lambda)\hat{\Phi} & = \frac{1+\eta_x^2}{(1+\eta)^2}\hat{\psi}_s\hat{\Phi}_s
        \end{align*}
        for $\mc{F}_1$, where restrictions to $S$ are implied, and similarly
        \begin{align*}
            D_\eta \mc{F}_2(w,\Lambda)H                     & =\begin{multlined}[t]
                \parn*{\frac{s\eta_{xx}\hat{\psi}_s}{(1+\eta)^2}-\frac{4s\eta_x^2\hat{\psi}_s}{(1+\eta)^3} + \frac{2s\eta_x\hat{\psi}_{xs}}{(1+\eta)^2}-2\frac{1+s^2\eta_x^2}{(1+\eta)^3}\hat{\psi}_{ss}}H\\
                + \parn*{\frac{4s\eta_x\hat{\psi}_s}{(1+\eta)^2}-\frac{2s\hat{\psi}_{xs}}{1+\eta} + \frac{2s^2 \eta_x\hat{\psi}_{ss}}{(1+\eta)^2}}H_x - \frac{s\hat{\psi}_s}{1+\eta} H_{xx}
            \end{multlined}                                                                                                                 \\
            D_{\hat{\varphi}} \mc{F}_2(w,\Lambda)\hat{\Phi} & = \parn*{\partial_x - \frac{s\eta_x}{1+\eta}\partial_s}^2 \hat{\Phi} + \frac{1}{(1+\eta)^2}\hat{\Phi}_{ss} + \gamma'(\hat{\psi}) \hat{\Phi}
        \end{align*}
        for $\mc{F}_2$. These expressions are of course valid for any $(w,\Lambda) \in \mc{O}$, regardless of whether this pair is a solution of \eqref{eq:F_zero}, but appear quite formidable. It turns out that \emph{if} $(w,\Lambda)$ is a solution of \eqref{eq:F_zero}, then $D_w\mc{F}(w,\Lambda)$ can be transformed into an operator that is easier to study.

        To that end, let us introduce the space
        \[
            X_2 \coloneqq \{\Phi \in C_{\kappa,\textnormal{e}}^{2,\beta}(\overline{\Omega}) : \Phi|_B=0\},
        \]
        which differs from $\hat{X}_2$, see \eqref{def:X}, only through the relaxation of the Dirichlet condition on $S$.
        The purpose of introducing this space is to ``encode'' both $H$ (i.e. capital $\eta$) and $\hat{\Phi}$ in a single variable.

        If $(w,\Lambda) \in \mc{O}$ is such that also $\eta \in C_{\kappa,\textnormal{e}}^{3,\beta}(\R)$ and $\hat{\varphi} \in C_{\kappa,\textnormal{e}}^{3,\beta}(\overline{\Omega})$, we may define a --- soon to be motivated --- bounded linear operator $\mc{L}(w,\Lambda) \in \lin(X_2,Y)$ by
        \begin{equation}
            \label{eq:L_definition}
            \begin{aligned}
                \mc{L}_1(w,\Lambda)\Phi & \coloneqq \frac{1+\eta_x^2}{(1+\eta)^2}\hat{\psi}_s \Phi_s + \parn*{\gamma(\mu) - \frac{1+\eta}{\hat{\psi}_s}}\Phi - \parn*{\frac{\eta_x\hat{\psi}_s}{1+\eta}\Phi}_x \\
                \mc{L}_2(w,\Lambda)\Phi & \coloneqq D_{\hat{\varphi}}\mc{F}_2(w,\Lambda)\Phi,
            \end{aligned}
        \end{equation}
        with the functions in the definition of $\mc{L}_1(w,\Lambda)$ evaluated on $S$. Furthermore, for $\mc{L}_2(w,\Lambda)$, we interpret $D_{\hat{\varphi}}\mc{F}_2(w,\Lambda)$ as extended to $X_2 \supset \hat{X}_2$ in the natural way. For convenience, we note that the operator defined in \eqref{eq:L_definition} simplifies to
        \begin{equation}
            \label{def:L_trivial}
            \begin{aligned}
                \mc{L}_1(\Lambda) & \coloneqq \mc{L}_1(0,\Lambda)\Phi = \lambda\Phi_s + \parn*{\gamma(\mu) - \frac{1}{\lambda}}\Phi \\
                \mc{L}_2(\Lambda) & \coloneqq \mc{L}_2(0,\Lambda)\Phi = (\Delta+\gamma'(\overline{\psi}))\Phi
            \end{aligned}
        \end{equation}
        at the trivial solutions.

        In particular, we have the required increased regularity for $\mc{L}(w,\Lambda)$ to be well defined when $(w,\Lambda)$ is a solution of \eqref{eq:F_zero}, due to \Cref{thm:analyticity_of_solutions}. Moreover, in this case we can relate $D_w\mc{F}(w,\Lambda)$ to the operator $\mc{L}(w,\Lambda)$, which is the reason for its introduction. This is done by utilizing a suitable generalization of the so-called $\mc{T}$-isomorphism from \cite{Ehrnstroem11Steady}.

        \begin{theorem}[$\mc{T}$-isomorphism]
            \label{thm:T_isomorphism}
            Suppose that $(w,\Lambda) \in \mc{O}$ satisfies \eqref{eq:F_zero}. Then
            \begin{equation}
                \label{eq:T_isomorphism}
                \mc{T}(w,\Lambda)\Phi \coloneqq \parn*{-\frac{1+\eta}{\hat{\psi}_s|_S}\Phi|_S,\Phi - \frac{s\hat{\psi}_s}{\hat{\psi}_s|_S} \Phi|_S}.
            \end{equation}
            defines an isomorphism $\map{\mc{T}(w,\Lambda)}{X_2}{X}$, and
            \begin{equation}
                \label{eq:L_DF_relationship}
                \mc{L}(w,\Lambda) = D_w\mc{F}(w,\Lambda)\mc{T}(w,\Lambda).
            \end{equation}
        \end{theorem}
        \begin{proof}
            Inspired by the procedure that was presumably used to arrive at the $\mc{T}$-isomorphism for trivial solutions in \cite{Ehrnstroem11Steady}, we suppose the existence of some element $\hat{f} \in X_2$ which enjoys the property
            \begin{equation}
                \label{eq:interesting_equation}
                D_{\hat{\varphi}}\mc{F}_2(w,\Lambda)\parn[\big]{\hat{f}H} = -D_\eta \mc{F}_2(w,\Lambda)H
            \end{equation}
            for all $H \in C^{2,\beta}_{\kappa,\textnormal{e}}(\R)$, and which furthermore does not vanish at any point on the surface. Again, we view the derivative $D_{\hat{\varphi}}\mc{F}_2(w,\Lambda)$ as naturally extended to an operator on $X_2 \supset \hat{X}_2$. Let $g \coloneqq \hat{f}|_S$, which by supposition has a definite sign. Then the operator $\map{\mc{T}(w,\Lambda)}{X_2}{X}$ defined by
            \begin{equation}
                \label{eq:T_isomorphism_general}
                \mc{T}(w,\Lambda)\Phi \coloneqq \parn*{-\frac{\Phi|_S}{g},\Phi - \hat{f}\frac{\Phi|_S}{g}}
            \end{equation}
            is easily seen to be an isomorphism, and yields
            \begin{align*}
                D_w \mc{F}_2(w,\Lambda)\mc{T}(w,\Lambda)\Phi & = -D_\eta \mc{F}_2(w,\Lambda)\parn*{\frac{\Phi|_S}{g}} + D_{\hat{\varphi}} \mc{F}_2(w,\Lambda)\parn*{\Phi-\hat{f}\frac{\Phi|_S}{g}}               \\
                                                             & =D_{\hat{\varphi}} \mc{F}_2(w,\Lambda)\parn*{\hat{f}\frac{\Phi|_S}{g}}+D_{\hat{\varphi}} \mc{F}_2(w,\Lambda)\parn*{\Phi-\hat{f}\frac{\Phi|_S}{g}} \\
                                                             & =D_{\hat{\varphi}} \mc{F}_2(w,\Lambda)\Phi,
            \end{align*}
            whence the second component of \eqref{eq:L_DF_relationship} is satisfied.

            We have shown that the property in \eqref{eq:interesting_equation} is key to establishing the theorem, and this equation turns out to be simpler to consider on the unflattened $\Omega_\eta$ instead. Define therefore the pullbacks $\psi = \hat{\psi} \circ \Pi$ and $f = \hat{f} \circ \Pi$, where we recall that $\Pi$ is the flattening transform from \eqref{def:flattening_transform}. Then the left-hand side of \eqref{eq:interesting_equation} becomes
            \begin{equation}
                \label{eq:interesting_left_hand_side}
                (\Delta + \gamma'(\psi))(fH) = (\Delta + \gamma'(\psi))fH + 2f_xH_x + fH_{xx},
            \end{equation}
            while the right-hand side turns into
            \begin{equation}
                \label{eq:interesting_right_hand_side}
                \parn*{\frac{2\eta_x^2 y \psi_y}{(1+\eta)^3} + \frac{2\psi_{yy}}{1+\eta} - \frac{y\eta_{xx}\psi_y}{(1+\eta)^2} - \frac{2\eta_xy\psi_{xy}}{(1+\eta)^2}}H\\ + 2\parn*{\frac{y\psi_y}{1+\eta}}_xH_x + \frac{y\psi_y}{1+\eta}H_{xx},
            \end{equation}
            after a lengthy computation. By comparing \eqref{eq:interesting_left_hand_side} and \eqref{eq:interesting_right_hand_side}, we see that the only possible solution candidate of \eqref{eq:interesting_equation} corresponds to
            \[
                f = \frac{y\psi_y}{1+\eta},
            \]
            and that we need only verify that the coefficients in front of $H$ in \eqref{eq:interesting_left_hand_side} and \eqref{eq:interesting_right_hand_side} are equal. Note that $f$ has the correct regularity for $\mc{T}$ to be well-defined, because $\psi$ and $\eta$ are analytic by \Cref{thm:analyticity_of_solutions}. Additionally, $\psi_y$ does not vanish on the surface, because of the postulation that $(w,\Lambda) \in \mc{O}$.

            One may now check by direct calculation that
            \[
                (\Delta + \gamma'(\psi))\parn*{\frac{y\psi_y}{1+\eta}} = \begin{multlined}[t]\frac{2\eta_x^2 y \psi_y}{(1+\eta)^3} + \frac{2\psi_{yy}}{1+\eta} - \frac{y\eta_{xx}\psi_y}{(1+\eta)^2} - \frac{2\eta_xy\psi_{xy}}{(1+\eta)^2} \\ +\frac{y}{1+\eta}(\Delta \psi + \gamma(\psi))_y,\end{multlined}
            \]
            where the first terms are precisely the ones in front of $H$ in \eqref{eq:interesting_right_hand_side}, while the last term vanishes because $\psi$ solves \eqref{eq:helmholtz}. Hence \eqref{eq:interesting_equation} holds for
            \[
                \hat{f} = f \circ \Pi^{-1} = \frac{s\hat{\psi}_s}{1+\eta},
            \]
            for which \eqref{eq:T_isomorphism_general} becomes \eqref{eq:T_isomorphism}. Finally, direct computation yields
            \[
                D_w\mc{F}_1(w,\Lambda)\mc{T}(w,\Lambda) = \begin{multlined}[t]\frac{1+\eta_x^2}{(1+\eta)^2}\hat{\psi}_s \Phi_s - \frac{\eta_x\hat{\psi}_s}{1+\eta}\Phi_x\\-\parn*{\frac{1+\eta}{\hat{\psi}_s} + \frac{1+\eta_x^2}{(1+\eta)^2}\hat{\psi}_{ss} + \frac{\eta_x^2\hat{\psi}_s}{(1+\eta)^2} - \frac{\eta_x\hat{\psi}_{xs}}{1+\eta}}\Phi,\end{multlined}
            \]
            where
            \[
                \frac{1+\eta_x^2}{(1+\eta)^2}\hat{\psi}_{ss} + \frac{\eta_x^2}{(1+\eta)^2}\hat{\psi} - \frac{\eta_x}{1+\eta}\hat{\psi}_{xs} = \parn*{\frac{\eta_x\hat{\psi}_s}{1+\eta}}_x -\gamma(\mu)
            \]
            on $S$ because $\mc{F}_2(w,\Lambda)=0$, and so \eqref{eq:L_DF_relationship} holds.
        \end{proof}

        \begin{remark}
            It is worth noting that if $\mc{F}_2(w,\Lambda)=0$, and we define the pullbacks $\psi = \hat{\psi} \circ \Pi$ and $\tilde{\Phi} = \Phi \circ \Pi$, then
            \begin{align*}
                \mc{L}_1(w,\Lambda)\Phi            & = \psi_y \partial^\perp \tilde{\Phi} - \parn*{\partial^\perp\psi_y + \frac{1}{\psi_y}}\tilde{\Phi}, \\
                (\mc{L}_2(w,\Lambda)\Phi)\circ \Pi & = (\Delta +\gamma'(\psi))\tilde{\Phi},
            \end{align*}
            where the functions in the expression for $\mc{L}_1$ are evaluated on $S_\eta$. By $\partial^\perp$, we here mean the non-normalized normal derivative $\partial^\perp \coloneqq \eta^\perp \cdot \nabla$ for $S_\eta$. Viewed through this lens, $\mc{L}(w,\Lambda)$ closely resembles the operator $\mc{L}(\Lambda)$ for the trivial solutions in \eqref{def:L_trivial}.

            We also mention that, while outside the scope of this paper, the $\mc{T}$-isomorphism can be employed even with the pseudo-stream function of waves in a stratified, incompressible fluid.
        \end{remark}

\section{Kernel and local bifurcation}
    \label{section:kernel_local_bifurcation}

    The purpose of this section is to describe the kernel of the operator $\mc{L}(\Lambda)$ from \eqref{def:L_trivial}, and to give the corresponding local bifurcation results for one-dimensional kernels. This extends parts of \cite{Ehrnstroem11Steady,Aasen18Traveling} to more general vorticity distributions, albeit with a slightly different bifurcation parameter. The paper \cite{Kozlov14Dispersion} deals with the same problem, in more detail, but with a quite different approach. Since our primary concern is global bifurcation, we present the results with this goal in mind. Note that \Cref{ass:gamma_regularity} is much stronger, especially the analyticity, than what is actually necessary for most of this section.

    To simplify the description of the kernel of $\mc{L}(\Lambda)$, we define $u = u(s;z)$ to be the solution of the initial value problem
    \begin{equation}
        \label{eq:u_initial_value_problem}
        \begin{gathered}
            u''(s;z) + (\gamma'(\overline{\psi}(s))-z)u(s;z) = 0,\\
            u(0;z) = 0, \qquad u'(0;z) = 1,
        \end{gathered}
    \end{equation}
    where $z$ acts as a parameter, and primes indicate derivatives with respect to $s$. Note that $u$ is entire in the parameter $z$ (see for instance \cite[Chapter 5]{Teschl12Ordinary}), and that $u$ also has a suppressed dependence on $\Lambda$ through $\overline{\psi}$. For real values of $z$ we may introduce the corresponding Prüfer angle $\vartheta=\vartheta(s;z)$ as the unique continuous representative of $\arg(u'+iu)$ with $\vartheta(0;z)=\arg(1)=0$. This representative is well defined since $u$ and $u'$ cannot vanish simultaneously, due to $u$ being the solution of \eqref{eq:u_initial_value_problem}. The Prüfer angle satisfies the first order equation
    \[
        \vartheta'(s;z) = \cos(\vartheta(s;z))^2 +(\gamma'(\overline{\psi}(s))-z)\sin(\vartheta(s;z))^2,
    \]
    at the cost of this equation being nonlinear.

    Recall that the derivative of $\gamma$ is bounded by \Cref{ass:gamma_regularity}. To facilitate the remainder of this section, we introduce the two quantities
    \[
        \rho \coloneqq \inf{\gamma'} \quad \text{and} \quad R \coloneqq \sup{\gamma'},
    \]
    as they are ubiquitous. The next lemma describes the behavior of the Prüfer angle $\vartheta(1;z)$ with respect to the parameter $z$. We will encounter this angle while describing the kernel.

    \begin{lemma}[Properties of $\vartheta$]
        \label{lemma:properties_of_prufer_angle}
        The Prüfer angle $\vartheta(1;\wildcard)$ is strictly decreasing, and in fact $\vartheta_z(1;\wildcard) < 0$. Moreover, it satisfies the bounds
        \begin{equation}
            \label{eq:bounds_for_the_prufer_angle}
            \sigma(z-\rho) \leq \vartheta(1;z) \leq \sigma(z-R),
        \end{equation}
        for all $z \in \R$, where $\map{\sigma}{\R}{(0,\infty)}$ is the (single-valued) function defined by
        \[
            \sigma(z) = \arg\parn*{\cosh(\sqrt{z})+i\frac{\sinh(\sqrt{z})}{\sqrt{z}}},
        \]
        with $\sigma(0) = \pi/4$. In particular, $\vartheta(1,-\infty) = \infty$ and $\vartheta(1,\infty) = 0$ (in the sense of limits).
    \end{lemma}
    \begin{proof}
        By differentiating \eqref{eq:u_initial_value_problem} with respect to $z$, multiplying by $u$ and integrating by parts, one arrives at the identity
        \[
            u_z'(1;z)u(1;z)-u_z(1;z)u'(1;z) = \int_0^1 u(s;z)^2\,ds,
        \]
        which implies that
        \begin{align*}
            \vartheta_z(1;z) & = \frac{u_z(1;z)u'(1;z)-u(1;z)u_z'(1;z)}{u(1;z)^2+u'(1;z)^2} \\
                             & = -\int_0^1 \frac{u(s;z)^2\,ds}{u(1;z)^2+u'(1;z)^2} < 0,
        \end{align*}
        proving the first part of the proposition.

        Define now $\tilde{u} = \tilde{u}(s;z)$ by
        \[
            \tilde{u}(s;z) \coloneqq \frac{\sinh(s\sqrt{z})}{\sqrt{z}},
        \]
        and $\tilde{\sigma}=\tilde{\sigma}(s;z)$ by
        \[
            \tilde{\sigma}(s;z) \coloneqq \arg(\tilde{u}'(s;z)+i\tilde{u}(s;z)),
        \]
        choosing the representative in the same way we did for $\vartheta$. To obtain the bounds described in \eqref{eq:bounds_for_the_prufer_angle}, it suffices to observe that $\tilde{\sigma}(0;z)=0$, and that the differential inequalities
        \begin{align*}
            \tilde{\sigma}'(s;z-R)    & \geq \cos(\tilde{\sigma}(s;z-R))^2 + (\gamma'(\overline{\psi}(s))-z)\sin(\tilde{\sigma}(s;z-R))^2       \\
            \intertext{and}
            \tilde{\sigma}'(s;z-\rho) & \leq \cos(\tilde{\sigma}(s;z-\rho))^2 + (\gamma'(\overline{\psi}(s))-z)\sin(\tilde{\sigma}(s;z-\rho))^2
        \end{align*}
        hold by direct computation. Then
        \[
            \tilde{\sigma}(s;z-\rho) \leq \vartheta(s;z) \leq \tilde{\sigma}(s;z-R)
        \]
        for all $s\geq 0$, and in particular we have \eqref{eq:bounds_for_the_prufer_angle} from the special case $s=1$.
    \end{proof}

    Since the function $u$ introduced in \eqref{eq:u_initial_value_problem} is entire in the parameter $z$, we may define a function $l$ by
    \begin{equation}
        \label{eq:kernel_l_definition}
        l(z,\Lambda) \coloneqq \frac{u'(1;z)}{u(1;z)},
    \end{equation}
    which consequently is meromorphic in $z$. Observing that $l(z,\Lambda) = \cot(\vartheta(1;z))$ on the real axis, we immediately obtain the following from \Cref{lemma:properties_of_prufer_angle}:

    \begin{figure}[htb]
        \centering
        \begin{tikzpicture}
	\pgfplotsset{compat=1.12, clip bounding box=default tikz}
	\begin{axis}[xmin=-50,xmax=100,ymin=-40,ymax=40,restrict y to domain=-300:300,xlabel=$z$,xtick={0},ytick={0},width=\textwidth, height=0.5\textwidth, no markers]
		\addplot[thick, black] table[header=false] {graphics/l0.bin};
		\addplot[thick, black] table[header=false] {graphics/l1.bin};
		\addplot[thick, black] table[header=false] {graphics/l2.bin};
		\addplot[thick, black] table[header=false] {graphics/l3.bin};
		
		%Need some special handling to handle the hyperbolic to trigonometric transition
		\addplot[name path=A, thick, black, dashed] table[header=false] {graphics/l0_lb.bin};
		\addplot[name path=B, thick, black, dashed, domain=(55-pi^2):100] {sqrt(x-45)/tanh(sqrt(x-45))};
		\addplot[gray!30] fill between[of=A and B];
		\addplot[black] coordinates {(55-pi^2,-40)(55-pi^2,40)};
		
		\addplot[thick, gray, dashed, domain=(45-pi^2):(55-pi^2)] {sqrt(55-x)/tan(deg(sqrt(55-x)))};
		\addplot[thick, gray, dashed, domain=(45-pi^2):45] {sqrt(45-x)/tan(deg(sqrt(45-x)))};
		
		\addplot[name path=C,thick, black, dashed, domain=(55-4*pi^2):(45-pi^2), samples=100] {sqrt(55-x)/tan(deg(sqrt(55-x)))};
		\addplot[name path=D,thick, black, dashed, domain=(55-4*pi^2):(45-pi^2), samples=100] {sqrt(45-x)/tan(deg(sqrt(45-x)))};
		\addplot[gray!30] fill between[of=C and D];
		\addplot[black] coordinates {(55-4*pi^2,-40)(55-4*pi^2,40)};
		\addplot[black] coordinates {(45-pi^2,-40)(45-pi^2,40)};
		
		\addplot[thick, gray, dashed, domain=(45-4*pi^2):(55-4*pi^2)] {sqrt(55-x)/tan(deg(sqrt(55-x)))};
		\addplot[thick, gray, dashed, domain=(45-4*pi^2):(55-4*pi^2)] {sqrt(45-x)/tan(deg(sqrt(45-x)))};
		
		\addplot[name path=E,thick, black, dashed, domain=(55-9*pi^2):(45-4*pi^2), samples=100] {sqrt(55-x)/tan(deg(sqrt(55-x)))};
		\addplot[name path=F,thick, black, dashed, domain=(55-9*pi^2):(45-4*pi^2), samples=100] {sqrt(45-x)/tan(deg(sqrt(45-x)))};
		\addplot[gray!30] fill between[of=E and F];
		\addplot[black] coordinates {(55-9*pi^2,-40)(55-9*pi^2,40)};
		\addplot[black] coordinates {(45-4*pi^2,-40)(45-4*pi^2,40)};
		
		\addplot[thick, gray, dashed, domain=(45-9*pi^2):(55-9*pi^2)] {sqrt(55-x)/tan(deg(sqrt(55-x)))};
		\addplot[thick, gray, dashed, domain=(45-9*pi^2):(55-9*pi^2)] {sqrt(45-x)/tan(deg(sqrt(45-x)))};
		
		\addplot[name path=G,thick, black, dashed, domain=(55-16*pi^2):(45-9*pi^2), samples=100] {sqrt(55-x)/tan(deg(sqrt(55-x)))};
		\addplot[name path=H,thick, black, dashed, domain=(55-16*pi^2):(45-9*pi^2), samples=100] {sqrt(45-x)/tan(deg(sqrt(45-x)))};
		\addplot[gray!30] fill between[of=G and H];
		\addplot[black] coordinates {(45-9*pi^2,-40)(45-9*pi^2,40)};
	\end{axis}
\end{tikzpicture}
        \caption{The graph of a particular instance of $l(\wildcard,\Lambda)$, with the lower and upper bounds furnished by \Cref{prop:properties_of_l}.}
        \label{fig:kernel_equation}
    \end{figure}
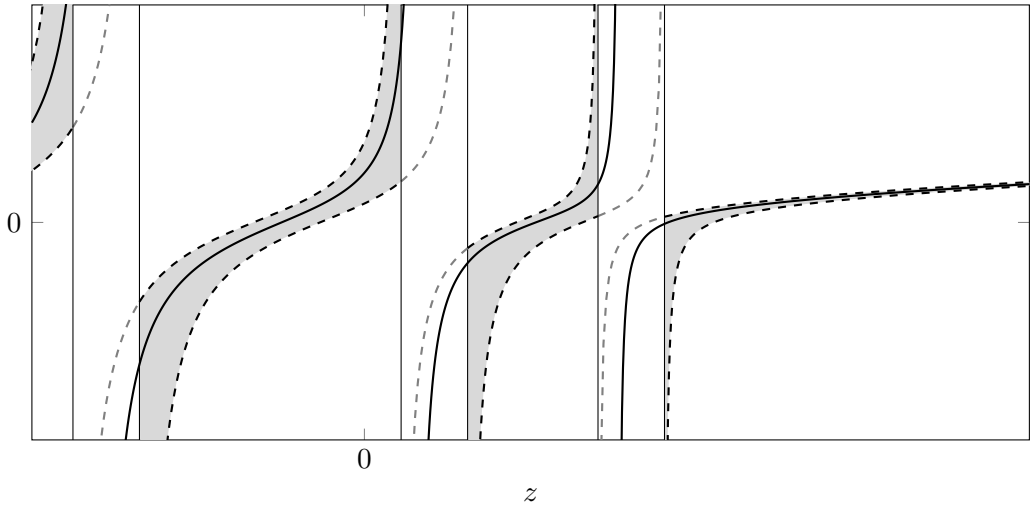

    \begin{proposition}[Properties of $l$]
        \label{prop:properties_of_l}
        The derivative of $l$ is positive on the real axis, except at the poles of $l$, which are all simple. Moreover, $l$ satisfies the bounds
        \begin{equation}
            \label{eq:bounds_for_l}
            v(z-R) \leq l(z,\Lambda) \leq v(z-\rho),
        \end{equation}
        where
        \[
            v(z)\coloneqq\cot(\sigma(z))=\frac{\sqrt{z}}{\tanh(\sqrt{z})},
        \]
        on the (possibly empty) intervals
        \[
            I_j \coloneqq \begin{cases}
                (R-(j+1)^2\pi^2,\rho-j^2\pi^2) & j \geq 1 \\
                (R-\pi^2,\infty)               & j=0.
            \end{cases}
        \]
        Finally, the point $z=0$ is a pole of $l$ if and only if $\overline{\psi}_\lambda(0) = 0$, and if not then
        \begin{equation}
            \label{eq:value_of_l_at_zero}
            l(0,\Lambda) = - \frac{\overline{\psi}_\mu(0)}{\overline{\psi}_\lambda(0)}.
        \end{equation}
    \end{proposition}
    \begin{proof}
        In order to show that \eqref{eq:value_of_l_at_zero} holds, it suffices to observe that
        \[
            u(s;0) = \overline{\psi}_\mu(0)\overline{\psi}_\lambda(s) - \overline{\psi}_\lambda(0)\overline{\psi}_\mu(s)
        \]
        for all $s \in [0,1]$. Indeed, the right-hand side satisfies \eqref{eq:u_initial_value_problem} (with $z=0$) by differentiation of \eqref{eq:trivial_solution_definition}, and vanishes at $s=0$. It also has the correct derivative at $s=0$ since
        \[
            \overline{\psi}_\mu(0)\overline{\psi}_\lambda'(0) - \overline{\psi}_\lambda(0)\overline{\psi}_\mu'(0) = \overline{\psi}_\mu(1)\overline{\psi}_\lambda'(1) - \overline{\psi}_\lambda(1)\overline{\psi}_\mu'(1) = 1,
        \]
        where we have used that the Wronskian of $\overline{\psi}_\mu$ and $\overline{\psi}_\lambda$ is constant.
    \end{proof}

    We are now equipped with everything we need to describe the kernel of the operator $\mc{L}(\Lambda)$.

    \begin{theorem}[Kernel of $\mc{L}(\Lambda)$]
        \label{thm:kernel}
        Let $\Lambda \in \mc{U}$. A basis for $\ker{\mc{L}(\Lambda)}$ is then given by $\{\Phi_n\}_{n \in M}$, where
        \[
            \Phi_n(x,s) \coloneqq \cos(n\kappa x)u(s;n^2\kappa^2)
        \]
        and $M = M(\Lambda)$ is the finite set of all $n \in \N_0$ satisfying the kernel equation
        \begin{equation}
            \label{eq:kernel_equation}
            l(n^2\kappa^2,\Lambda) = r(\Lambda),
        \end{equation}
        where
        \begin{equation}
            \label{eq:kernel_r_definition}
            r(\Lambda) \coloneqq \frac{1}{\lambda^2} - \frac{\gamma(\mu)}{\lambda},
        \end{equation}
        and $l$ is the function defined in \eqref{eq:kernel_l_definition}.
    \end{theorem}
    \begin{proof}
        Suppose that $\Phi \in X_2$, and write it as a Fourier series
        \[
            \Phi(x,s) = \sum_{n=0}^\infty a_n(s)\cos(n\kappa x)
        \]
        in the horizontal direction. By inserting the series into \eqref{def:L_trivial}, we deduce that $\Phi \in \ker{\mc{L}(\Lambda)}$ if and only if each coefficient $a_n$ solves the regular Sturm--Liouville problem
        \begin{gather*}
            a_n''(s) + (\gamma'(\overline{\psi}(s))-n^2 \kappa^2)a_n(s)=0\\
            a_n(0)=0, \qquad \lambda a_n'(1) + \parn*{\gamma(\mu)-\frac{1}{\lambda}}a_n(1)=0
        \end{gather*}
        for all $n \in \N_0$. Trivially, $a_n = 0$ is always a solution, but not always the only one: It is well known that this Sturm--Liouville problem has nonzero solutions, spanned by $u(\wildcard;n^2\kappa^2)$, if and only if \eqref{eq:kernel_equation} is satisfied. There are only finitely many solutions of \eqref{eq:kernel_equation}, as $n^2\kappa^2 \in I_0$ for all sufficiently large $n \in \N_0$, and $l$ is strictly increasing there.
    \end{proof}
    \begin{remark}
        The function $l$ will depend non-trivially on $\Lambda$ unless $\gamma'$ is a constant, namely when the vorticity is either constant or affine. We also mention that we would typically like to avoid the degenerate case where $n = 0$ solves \eqref{eq:kernel_equation}. For this reason, \eqref{eq:value_of_l_at_zero} can occasionally be useful.
    \end{remark}

    Exactly one-dimensional kernels can be found under certain assumptions on $\gamma$, especially if we are willing to relinquish control of the wavenumber $\kappa > 0$. One way is to, in essence, require that $\gamma$ is close enough to affine to enable us to exploit the bounds in \eqref{eq:bounds_for_l}. Note that there is no loss of generality in limiting the scope to $M(\Lambda)=\{1\}$, by redefining $\kappa$, as long as the only interest is in one-dimensional kernels $M(\Lambda) \neq \{0\}$.

    \begin{proposition}[Kernel construction]\leavevmode
        \label{thm:kernel_construction}
        \begin{enumerate}[(i)]
            \item Suppose that $n^2\kappa^2 \in I_j$ for some $j \in \N_0$, and further that $\mu$ is such that
                \[
                    \gamma(\mu)^2 >- 4v(n^2\kappa^2-R),
                \]
                with $v$ as defined in \Cref{prop:properties_of_l}. Then there exist $\lambda \neq 0$ such that $n\in M(\Lambda)$. More precisely, such $\lambda$ can always be chosen to satisfy
                \begin{equation}
                    \label{eq:lambda_intervals_nonpositive}
                    \lambda \in \begin{cases}
                        (0,-2/\gamma(\mu)) & \gamma(\mu) < 0, \\
                        (-2/\gamma(\mu),0) & \gamma(\mu) > 0,
                    \end{cases}
                \end{equation}
                if $v(n^2\kappa^2 - R) \leq 0$, and
                \begin{equation}
                    \label{eq:lambda_intervals_positive}
                    \lambda \in \begin{cases}
                        \text{$(-\infty,0)$ or $(0,-1/\gamma(\mu))$} & \gamma(\mu) < 0 \\
                        \text{$(-\infty,0)$ or $(0,\infty)$}         & \gamma(\mu) = 0 \\
                        \text{$(-1/\gamma(\mu),0)$ or $(0,\infty)$}  & \gamma(\mu) > 0
                    \end{cases}
                \end{equation}
                otherwise.
            \item Assume that $0 \in I_j$ for some $j \in \N_0$, and let $\mu \in \R$. For any $\kappa > 0$ such that $\kappa^2 \in I_0$ and
                \[
                    v(\kappa^2-R) > \max(v(-\rho),-\gamma(\mu)^2/4)
                \]
                there is a $\lambda \neq 0$, which can be chosen according to  \eqref{eq:lambda_intervals_nonpositive} or \eqref{eq:lambda_intervals_positive}, such that $M(\Lambda) \equiv \{1\}$. In particular, this is the case for all sufficiently large $\kappa > 0$.
        \end{enumerate}
    \end{proposition}
    \begin{proof}
        We know by \Cref{prop:properties_of_l} that if $n^2\kappa^2 \in I_j$, then
        \[
            v(n^2\kappa^2-R) \leq l(n^2\kappa^2,\Lambda) \leq v(n^2\kappa^2-\rho),
        \]
        where it is crucial that the bounds do not depend on $\Lambda$. Observe that
        \[
            \inf_{\lambda \neq 0}{r(\mu,\lambda)} = - \frac{1}{4}\gamma(\mu)^2
        \]
        for every $\mu \in \R$, with the infimum attained at $\lambda = -2/\gamma(\mu)$ as long as $\gamma(\mu) \neq 0$. In the same event, we also have $r(\mu,-1/\gamma(\mu))=0$. Moreover, $r(\mu,\lambda) \to \infty$ as $\lambda \to 0$, and $r(\mu,\lambda) \to 0$ as $\abs{\lambda}\to \infty$. The first part of the proposition now follows from the intermediate value theorem applied to $r(\mu,\wildcard)-l(n^2\kappa^2,\mu,\wildcard)$, on appropriate intervals chosen according to either \eqref{eq:lambda_intervals_nonpositive} or \eqref{eq:lambda_intervals_positive}.

        For the second part of the proposition, observe that the hypothesis of the first part is satisfied with $n=1$ and $j=0$. Thus, there is some $\lambda \neq 0$ such that $1 \in M(\Lambda)$. Moreover, by the assumptions and \Cref{prop:properties_of_l} we have
        \[
            l(0,\Lambda) \leq v(-\rho) < v(\kappa^2 - R) \leq l(\kappa^2,\Lambda),
        \]
        whence $0 \notin M(\Lambda)$. Finally, $l(n^2\kappa^2,\Lambda) > l(\kappa^2,\Lambda)$ for all $n \geq 2$, since $l$ is strictly increasing on $I_0$. Thus $M(\Lambda) = \{1\}$.
    \end{proof}
    \begin{remark}
        Kernels of arbitrarily large finite dimension exist when $\gamma$ is affine \cite{Aasen18Traveling}, but it is unclear if, and in what sense, the existence of multi-dimensional kernels generalizes to more general vorticity distributions. We will not pursue this question here.
    \end{remark}

    \subsection{The Fredholm property of \texorpdfstring{$\mc{L}(\Lambda)$}{L}}

        Let us introduce the notation
        \[
            \mc{T}(\Lambda) = (\mc{T}_1,\mc{T}_2)(\Lambda) \coloneqq \mc{T}(0,\Lambda), \qquad \Lambda \in \mc{U},
        \]
        for the $\mc{T}$-isomorphism at the trivial solutions, mirroring our use of $\mc{L}(\Lambda)$ in \eqref{def:L_trivial}. If we also equip $Y$ with the inner product (inducing a finer topology)
        \begin{equation}
            \label{eq:Y_inner_product}
            \angl{w_1,w_2}_Y \coloneqq \angl{\eta_1,\eta_2}_{L^2_\kappa(\R)} + \angl{\hat{\varphi}_1,\hat{\varphi}_2}_{L^2_\kappa(\Omega)}, \qquad w_i = (\eta_i,\hat{\varphi}_i),
        \end{equation}
        we can state a useful lemma. In particular, we will employ it to describe the image of $\mc{L}(\Lambda)$.

        \begin{lemma}[Symmetry for $\mc{L}(\Lambda)$]
            \label{lemma:L_symmetry}
            The identity
            \[
                \angl*{\parn*{\mc{T}_1(\Lambda)\Phi,\Phi},\mc{L}(\Lambda)\Psi}_Y = \angl*{\mc{L}(\Lambda)\Phi,\parn*{\mc{T}_1(\Lambda)\Psi,\Psi}}_Y
            \]
            holds for all $\Phi, \Psi \in X_2$.
        \end{lemma}
        \begin{proof}
            By one of Green's identities, we have
            \[
                \int_\Omega\parn*{\Phi \Delta \Psi - \Psi \Delta \Phi}\,dx\,ds = \int_S\parn*{\Phi \Psi_s - \Phi_s \Psi}\,dx,
            \]
            where the integrals are understood to be over one period. Therefore
            \[
                \angl{\Phi,\mc{L}_2(\Lambda)\Psi}_{L_\kappa^2(\Omega)} = \angl{\mc{L}_2(\Lambda)\Phi,\Psi}_{L_\kappa^2(\Omega)} + \int_S\parn*{\Phi \Psi_s - \Phi_s \Psi}\,dx,
            \]
            and as a consequence, we find
            \begin{align*}
                \angl*{\parn*{\mc{T}_1(\Lambda)\Phi,\Phi},\mc{L}(\Lambda)\Psi}_Y & = \begin{multlined}[t]\angl{\mc{L}_2(\Lambda)\Phi,\Psi}_{L_\kappa^2(\Omega)} + \int_S\parn*{\Phi \Psi_s - \Phi_s \Psi}\,dx\\
                    + \int_S \Phi (r(\Lambda)\Psi - \Psi_s)\,dx
                \end{multlined}                                                                                                  \\
                                                                                 & =\angl{\mc{L}_2(\Lambda)\Phi,\Psi}_{L_\kappa^2(\Omega)} + \angl{\mc{L}_1(\Lambda)\Phi,\mc{T}_1(\Lambda)\Psi}_{L_\kappa^2(\R)} \\
                                                                                 & =\angl*{\mc{L}(\Lambda)\Phi,\parn*{\mc{T}_1(\Lambda)\Psi,\Psi}}_Y
            \end{align*}
            by direct computation.
        \end{proof}

        Since $\mc{L}(\Lambda)$ is a simple elliptic operator (with boundary conditions), it is a standard result that it, and by consequence $D_w\mc{F}(0,\Lambda)$ through \Cref{thm:T_isomorphism}, is Fredholm of index zero. Stated more precisely, we have the following:

        \begin{lemma}[Fredholm property of $\mc{L}(\Lambda)$]
            \label{lemma:fredholm_property_trivial}
            Suppose that $\Lambda \in \mc{U}$. Then $\mc{L}(\Lambda)$ is a Fredholm operator of index zero. Moreover, its image is the orthogonal complement of the subspace
            \[
                Z(\Lambda)\coloneqq\brac*{(\mc{T}_1(\Lambda)\Phi,\Phi): \Phi \in \ker{\mc{L}}(\Lambda)}
            \]
            in $Y$ with respect to the inner product in \eqref{eq:Y_inner_product}.
        \end{lemma}

        We omit the proof of \Cref{lemma:fredholm_property_trivial}, opting only to motivate the result by noting that the inclusion
        \[
            \img{\mc{L}(\Lambda)} \subset Z(\Lambda)^\perp
        \]
        is an immediate corollary of \Cref{lemma:L_symmetry}. The opposite inclusion is less trivial. See for instance \cite{Wahlen06Steady} for a proof in a similar setting.

    \subsection{Transversality and local bifurcation}

        An additional benefit of introducing the kernel equation \eqref{eq:kernel_equation} is that the transversality condition for local bifurcation, appearing in the Crandall--Rabinowitz theorem, can be expressed using a a differentiated version of this equation. To prove this, we exploit the characterization of $\img{\mc{L}(\Lambda)}$ given in \Cref{lemma:fredholm_property_trivial}.

        \begin{proposition}[Transversality condition]
            \label{prop:transversality_condition}
            Suppose that $\Lambda \in \mc{U}$, and that $M(\Lambda) = \{n\}$ for some $n \in \N_0$. Then
            \begin{equation}
                \label{eq:transversality_definition}
                D_{w\mu}\mc{F}(0,\Lambda)\mc{T}(\Lambda)\Phi_n \notin \img{D_w\mc{F}(0,\Lambda)}
            \end{equation}
            if and only if the transversality condition
            \begin{equation}
                \label{eq:transversality_condition}
                l_\mu(n^2\kappa^2,\Lambda) \neq r_\mu(\Lambda)
            \end{equation}
            is satisfied. Here, the functions $l$ and $r$ are those defined in \eqref{eq:kernel_l_definition} and \eqref{eq:kernel_r_definition}, respectively, and the subscripts denote partial derivatives.
        \end{proposition}
        \begin{proof}
            We first observe that by the identity
            \[
                D_{w\mu}\mc{F}(0,\Lambda)\mc{T}(\Lambda) + D_w\mc{F}(0,\Lambda)\mc{T}_\mu(\Lambda) = \mc{L}_\mu(\Lambda),
            \]
            which follows from \eqref{eq:L_DF_relationship}, we have that the condition
            \[
                \mc{L}_\mu(\Lambda)\Phi_n \notin \img{\mc{L}(\Lambda)}
            \]
            is equivalent to \eqref{eq:transversality_definition}. Further, this condition is, in turn, equivalent to
            \[
                \angl{\mc{L}_\mu(\Lambda)\Phi_n,(\mc{T}_1(\Lambda)\Phi_n,\Phi_n)}_Y \neq 0,
            \]
            or
            \begin{equation}
                \label{eq:transversality_proof_equivalent_condition}
                \int_0^1 \gamma''(\overline{\psi}(s))\overline{\psi}_\mu(s)u(s;n^2\kappa^2)^2\,ds \neq \frac{\gamma'(\mu)}{\lambda}u(1;n^2\kappa^2)^2,
            \end{equation}
            by \Cref{lemma:fredholm_property_trivial}.

            We immediately recognize that $\gamma'(\mu)/\lambda = -r_\mu(\Lambda)$ on the right-hand side of \eqref{eq:transversality_proof_equivalent_condition}. The result finally follows by observing that
            \[
                \int_0^1 \gamma''(\overline{\psi}(s))\overline{\psi}_\mu(s)u(s;n^2\kappa^2)^2\,ds = - u(1;n^2\kappa^2)^2 l_\mu(n^2\kappa^2,\Lambda),
            \]
            which is obtained by differentiating \eqref{eq:u_initial_value_problem} with respect to $\mu$, multiplying by $u(\wildcard;n^2\kappa^2)$, and integrating by parts with respect to $s$. Note that $u(1;n^2\kappa^2)$ is necessarily nonzero, since \eqref{eq:kernel_equation} is satisfied by hypothesis.
        \end{proof}
        \begin{remark}
            A completely analogous transversality condition to \eqref{eq:transversality_condition} holds if $\lambda$ is used as the bifurcation parameter instead of $\mu$. The only change needed is to exchange the partial derivatives for ones with respect to $\lambda$.
        \end{remark}

        We can now apply the Crandall--Rabinowitz theorem, see \cite{Crandall71Bifurcation} (or \cite{Buffoni03Analytic} for a more modern exposition), to obtain small-amplitude waves that solve \eqref{eq:F_zero}. This extends the corresponding theorem in \cite{Ehrnstroem11Steady} to more general vorticity distributions than affine. Recall that $\mc{U}$ and $\mc{O}_\lambda$ are sets of permissible parameters and solutions, respectively introduced in \eqref{def:permissible_parameters} and \eqref{def:open_set_slice}. Note that a similar result to \Cref{thm:local_bifurcation}, also for small-amplitude waves, has previously been obtained in \cite{Kozlov14Dispersion}.

        \begin{theorem}[Local bifurcation]
            \label{thm:local_bifurcation}
            Let $\Lambda^* \in \mc{U}$ and suppose that $M(\Lambda^*) = \{n\}$ for some $n \in \N$, so that
            \[
                \ker{D_w\mc{F}(0,\Lambda^*)} = \lspan{\{\mc{T}(\Lambda^*)\Phi_n\}},
            \]
            where $M$ and $\Phi_n$ are as in \Cref{thm:kernel}. If the transversality condition \eqref{eq:transversality_condition} holds, there exists an analytic curve $\mc{K}_{\Lambda^*}^{\textnormal{loc}}=\{(w(t),\mu(t)) : \abs{t} < \varepsilon\}$ of solutions to
            \begin{equation}
                \label{eq:F_zero_starred}
                \mc{F}(w,\mu,\lambda^*) = 0
            \end{equation}
            in $\mc{O}_{\lambda^*}$, with
            \begin{align}
                \notag
                w(t)   & = t\mc{T}(\Lambda^*)\Phi_n + O(t^2) \quad \text{(in $X$)}
                \intertext{and}
                \label{eq:mu_asymptotic}
                \mu(t) & = \mu^* + O(t^2)
            \end{align}
            as $t \to 0$. The solutions on the curve $\mc{K}_{\Lambda^*}^\textnormal{loc}$ have wavenumber $n\kappa$, and
            \begin{equation}
                \label{eq:curve_symmetry}
                \begin{aligned}
                    \mu(-t)                & = \mu(t),                                           \\
                    \eta(-t)(x)            & = \eta(t)\parn*{x + \frac{\pi}{n\kappa}},           \\
                    \hat{\varphi}(-t)(x,s) & = \hat{\varphi}(t)\parn*{x + \frac{\pi}{n\kappa},s}
                \end{aligned}
            \end{equation}
            for all $\abs{t} < \epsilon$ and $(x,s) \in \overline{\Omega}$.

            In addition, there is a neighborhood of $(0,\mu^*) \in \mc{O}_{\lambda^*}$ in which all solutions of \eqref{eq:F_zero_starred}  are either trivial or on the curve.
        \end{theorem}
        \begin{proof}
            The only parts of the theorem that do not follow directly from the Crandall--Rabinowitz theorem are:
            \begin{enumerate}[(i)]
                \item The claim that the solutions have wavenumber $n\kappa$, \label{point:local_bifurcation_wavenumber}
                \item the symmetry properties in \eqref{eq:curve_symmetry}, and \label{point:local_bifurcation_symmetry}
                \item the asymptotics in \eqref{eq:mu_asymptotic}. \label{point:local_bifurcation_asymptotic}
            \end{enumerate}

            Here, \cref{point:local_bifurcation_wavenumber} follows by redefining $\kappa$ such that $n=1$ before applying the Crandall--Rabinowitz theorem, while \cref{point:local_bifurcation_symmetry} can be obtained by observing that $(w,\mu)$ is a solution of \eqref{eq:F_zero_starred} if and only if
            \[
                \parn*{x \mapsto \eta\parn*{x + \frac{\pi}{n\kappa}},(x,s) \mapsto \hat{\varphi}\parn*{x+\frac{\pi}{n\kappa},s},\mu}.
            \]
            is a solution. Lastly, \cref{point:local_bifurcation_asymptotic} is an immediate corollary of the just-proved symmetry of $\mu$ in \ref{point:local_bifurcation_symmetry}.
        \end{proof}

\section{Explicit examples}
    \label{sec:explicit}

    We believe there is value in pausing to record the two simplest forms of vorticity distributions here, for which many aspects of the theory become significantly more explicit.
    \subsection{Constant vorticity}
        If the vorticity distribution is constant; that is, of the form
        \[
            \gamma(t) \equiv \omega_0,
        \]
        for some fixed $\omega_0 \in \R$, then the trivial solutions (solving \eqref{eq:trivial_solution_definition}) are given by the quadratic polynomials
        \[
            \overline{\psi}(s;\Lambda) = \mu + \lambda(s-1) - \frac{1}{2}\omega_0(s-1)^2
        \]
        for every $\Lambda \in \mc{U}$. Consequently
        \[
            \Upsilon(\Lambda) = \mu - \lambda - \frac{1}{2}\omega_0,
        \]
        and we see that
        \[
            u(s;z) = \tilde{u}(s;z) = \frac{\sinh(s\sqrt{z})}{\sqrt{z}},
        \]
        with the notation taken from the proof of \Cref{lemma:properties_of_prufer_angle}, solves \eqref{eq:u_initial_value_problem}. Hence
        \[
            l(z,\Lambda) = v(z)= \frac{\sqrt{z}}{\tanh(\sqrt{z})}
        \]
        in the kernel equation \eqref{eq:kernel_equation}. This can also obtained directly from \Cref{prop:properties_of_l} in this case, as $\rho = R = 0$.

        Bifurcation with respect to $\mu$ is never possible for constant vorticity, as \eqref{eq:transversality_condition} can \emph{never} be satisfied. Therefore \Cref{thm:local_bifurcation} does not apply, in the way it is stated here. This is not unexpected, as changing $\mu$ merely constitutes a constant shift of $\overline{\psi}$, and one may equally well set $\mu \equiv 0$. Bifurcation with respect to $\lambda$, on the other hand, can be done from either of the simple bifurcation points
        \begin{equation}
            \label{eq:constant_bifurcation_points}
            \frac{1}{\lambda_{n,\pm}} = \frac{\omega_0}{2} \pm \sqrt{\parn*{\frac{\omega_0}{2}}^2 + \frac{n\kappa}{\tanh(n\kappa)}},
        \end{equation}
        for any $n \in \N$. This is because the transversality condition with respect to $\lambda$ becomes $\omega_0 \lambda_{n,\pm} \neq 2$, which is always satisfied. Global bifurcation for the constant case, including stagnation, has already been studied in great detail in \cite{Constantin16Global}.

        The trivial solution corresponding to a bifurcation point in \eqref{eq:constant_bifurcation_points} exhibits stagnation if and only if $\omega_0 \neq 0$, the sign is chosen opposite that of $\omega_0$, and
        \[
            \frac{n\kappa}{\tanh(n\kappa)} \geq 1 + \frac{1}{\omega_0^2}.
        \]
        holds. This stagnation presents as a critical line of stagnation points at $s = 1 + \lambda_{n,\pm}/\omega_0$, opening up to a single critical layer of closed streamlines in each minimal period of nearby solutions on $\mc{K}_{\Lambda^*}^\textnormal{loc}$. See \cite[Theorem 4.1]{Wahlen09Steady} or \cite[Theorem 16]{Constantin16Global} for more details.

    \subsection{Affine vorticity}

        As discussed in \cite{Ehrnstroem11Steady}, it is sufficient to instead consider only \emph{linear} vorticity distributions of the form
        \[
            \gamma(t) = \omega_0 t
        \]
        for fixed $\omega_0 \neq 0$. The trivial solutions take the form
        \[
            \overline{\psi}(s;\Lambda) = \mu\cos\parn*{\sqrt{\omega_0}(s-1)} + \lambda \frac{\sin\parn*{\sqrt{\omega_0}(s-1)}}{\sqrt{\omega_0}},
        \]
        and are therefore trigonometric when $\omega_0 > 0$, and hyperbolic when $\omega_0 < 0$. Accordingly,
        \[
            \Upsilon(\Lambda) = \mu\cos\parn*{\sqrt{\omega_0}} + \lambda \frac{\sin\parn*{\sqrt{\omega_0}}}{\sqrt{\omega_0}},
        \]
        and
        \[
            u(s;z) = \tilde{u}(s;z-\omega_0), \qquad l(z,\Lambda) = v(z-\omega_0),
        \]
        with $\tilde{u}$ and $v$ as in the case of constant $\gamma$ above.

        Even with this simplest choice of non-constant vorticity distribution, there is a much richer structure of bifurcation points. Furthermore, the transversality condition for one-dimensional bifurcation is trivially satisfied for $\mu$, so \Cref{thm:local_bifurcation} does apply when the other hypotheses are met. We also mention that this condition reduces to
        \[
            \omega_0 \mu \lambda \neq 2
        \]
        when $\lambda$ is used as the bifurcation parameter. See the works \cite{Ehrnstroem11Steady,Aasen18Traveling,Ehrnstroem15Trimodal} for local bifurcation results with slightly different choices of parameters (which are quite hard to generalize to non-affine $\gamma$), including thorough exploration of the resulting kernel equation.

        A small computation shows that at most one critical line is present at a trivial solution when $\omega_0 > 0$, but that they can have any number of such lines if $\omega_0 < 0$ is sufficiently negative. At a simple bifurcation point, these open up to critical layers in nearby solutions on $\mc{K}_{\Lambda^*}^\textnormal{loc}$, just as for constant vorticity. This means that solutions on the local bifurcation curves can display arbitrarily many critical layers, as described in \cite{Ehrnstroem12Steady}.

\section{Global bifurcation}
    \label{section:global_bifurcation}
    Our local bifurcation result, \Cref{thm:local_bifurcation}, establishes the existence of \emph{local} curves of small solutions to \eqref{eq:F_zero}. We will now proceed to the main event of this paper, which is to use analytic global bifurcation theory, due to Buffoni, Dancer \& Toland \cite{Dancer73Bifurcation,Buffoni03Analytic}, to extend these local curves to \emph{global} curves. The principal result is the following theorem:

    \begin{theorem}[Global bifurcation]
        \label{thm:global_bifurcation}
        The local curve obtained in \Cref{thm:local_bifurcation} can be uniquely extended (up to reparametrization) to a continuous curve
        \[
            \mc{K}_{\Lambda^*} = \{(w(t),\mu(t)) : t \in \R\} \supset \mc{K}_{\Lambda^*}^\textnormal{loc}
        \]
        of solutions to \eqref{eq:F_zero_starred}, such that the following properties hold:
        \begin{enumerate}[(i)]
            \item The curve can be reparametrized analytically in a neighborhood of any point on the curve.
            \item The solutions have wavenumber $n\kappa$, and satisfy the symmetry properties \eqref{eq:curve_symmetry}, for all $t \in \R$.
            \item One of the following alternatives occur:
                \begin{enumerate}[(A)]
                    \item Either
                        \[
                            \min{\brac*{\frac{1}{1+\norm{w(t)}_{X} + \abs{\mu(t)}},\min_{x \in \R}{\parn*{1+\eta(t)}}, \min_{S}{\abs[\big]{\hat{\psi}_s(t)}}}} \to 0
                        \]
                        as $t \to \infty$, \label{item:curve_good_alternative}
                    \item or the curve is closed. \label{item:curve_closed}
                \end{enumerate}
        \end{enumerate}
    \end{theorem}
    \begin{remark}
        Alternative \ref{item:curve_good_alternative} would imply the existence of subsequences $(t_n)_{n \in \N}$, with $t_n \to \infty$, along which at least one of
        \begin{enumerate*}[(i)]
            \item the solutions are unbounded,
            \item the surface approaches the bed, or
            \item surface stagnation is approached,
        \end{enumerate*}
        hold true.

    \end{remark}

    \Cref{thm:global_bifurcation} will follow directly from a slightly modified version of \cite[Theorem 9.1.1]{Buffoni03Analytic}, stated in \cite[Theorem 6]{Constantin16Global}, if we can prove the required Fredholm and compactness properties. Namely, that:
    \begin{enumerate}[(i)]
        \item The derivative $D_w\mc{F}(w,\mu,\lambda^*)$ is Fredholm of index zero not only when $(w,\mu) = (0,\mu^*)$, but on the entire solution set of \eqref{eq:F_zero_starred}.
        \item For an appropriately chosen increasing sequence $(\mc{Q}^{\lambda^*}_j)_{j \in \N}$ of closed and bounded subsets of $\mc{O}_{\lambda^*}$ such that
            \[
                \mc{O}_{\lambda^*} = \bigcup_{j \in \N} \mc{Q}^{\lambda^*}_j,
            \]
            the intersection
            \[
                \brac*{(w,\mu) \in \mc{O}_{\lambda^*} : \mc{F}(w,\mu,\lambda^*)=0} \cap \mc{Q}^{\lambda^*}_j
            \]
            is compact in $X$ for each $j \in \N$.
    \end{enumerate}
    Thus, to establish \Cref{thm:global_bifurcation}, we are left to verify that these two conditions are satisfied. We have already made some of the necessary preparations for this in previous sections.

    \subsection{Verification of the global Fredholm property}

        The central tool we will use to show that the Fréchet derivative of $\mc{F}$ is Fredholm on the solution set of \eqref{eq:F_zero}, is the generalized $\mc{T}$-isomorphism from \Cref{thm:T_isomorphism}. Like for the trivial solutions, this reduces the problem to one for the simpler operator $\mc{L}$. First, we show semi-Fredholmness whenever $\mc{L}$ is well defined. Note that, importantly, $(w,\Lambda)$ need \emph{not} be a solution of \eqref{eq:F_zero} in the lemma.

        \begin{lemma}
            \label{lemma:semi_fredholm_L}
            Let $(w,\Lambda)\in \mc{O}$, with $\mc{O}$ as in \eqref{eq:open_set}, be such that also $\eta \in C_{\kappa,\textnormal{e}}^{3,\beta}(\R)$ and $\hat{\varphi} \in C_{\kappa,\textnormal{e}}^{3,\beta}(\overline{\Omega})$. Then $\mc{L}(w,\Lambda)$ has finite-dimensional kernel and closed range.
        \end{lemma}
        \begin{proof}
            Recall the definition of $\map{\mc{L}(w,\Lambda)}{X_2}{Y}$ in \eqref{eq:L_definition}. Since
            \begin{equation}
                \label{eq:uniform_ellipticity}
                \begin{aligned}
                    \xi_1^2 -\frac{2s\eta_x}{1+\eta}\xi_1\xi_2 + \frac{1+s^2\eta_x^2}{(1+\eta)^2}\xi_2^2 & = \xi^\intercal \begin{pmatrix}
                        1                        & - \dfrac{s\eta_x}{1+\eta}         \\[0.75em] %Need slightly more breathing room here
                        -\dfrac{s\eta_x}{1+\eta} & \dfrac{1+s^2\eta_x^2}{(1+\eta)^2}
                    \end{pmatrix}
                    \xi                                                                                                                                             \\
                                                                                                         & \geq \frac{1}{(1+\eta)^2 + 1 + s^2 \eta_x^2}\abs{\xi}^2,
                \end{aligned}
            \end{equation}
            for all $\xi= (\xi_1,\xi_2) \in \R^2$, the operator component $\mc{L}_2(w,\Lambda): X_2\to Y_1$ is strictly elliptic. The inequality in \eqref{eq:uniform_ellipticity} can be deduced from the eigenvalues of the matrix.

            Furthermore, the coefficient in front of $\partial_s$ in $\mc{L}_1(w,\Lambda)$ is uniformly separated away from $0$, by how we defined $\mc{O}$. Combining now the Schauder estimates from Theorems 6.6 and 6.30 in \cite{Gilbarg01Elliptic}, we deduce that there is a constant $C=C(w,\Lambda) \geq 0$ such that
            \begin{equation}
                \label{eq:Schauder_estimate_for_L}
                \norm{\Phi}_{X_2} \leq C(\norm{\Phi}_{L^\infty} + \norm{\mc{L}(w,\Lambda)\Phi}_Y)
            \end{equation}
            for all $\Phi \in X_2$. Standard arguments based on \eqref{eq:Schauder_estimate_for_L} can in turn be used to establish the lemma.
        \end{proof}

        Armed with \Cref{lemma:semi_fredholm_L}, we can prove the desired Fredholm property, by relating $\mc{L}(w,\Lambda)$ and $\mc{L}(\Lambda)$ and then employing the stability of the Fredholm index.

        \begin{theorem}[Global Fredholm property]
            Suppose that $(w,\Lambda) \in \mc{O}$ is a solution of \eqref{eq:F_zero}. Then the operator $D_w \mc{F}(w,\Lambda)$ is Fredholm of index zero.
        \end{theorem}
        \begin{proof}
            Recalling \Cref{rem:open_set}, we have $(tw,\Lambda) \in \mc{O}$ for every $t \in [0,1]$. By \Cref{thm:analyticity_of_solutions} used on $(w,\Lambda)$, the necessary regularity for \Cref{lemma:semi_fredholm_L} to apply is also present. Thus $\mc{L}(tw,\Lambda)$ has finite-dimensional kernel and closed range for each $t \in [0,1]$. This is the case even if $(tw,\Lambda)$ need not be a solution of \eqref{eq:F_zero} in general, except at the endpoints. In particular, the operators $\mc{L}(tw,\Lambda)$ are semi-Fredholm, so their Fredholm index is well defined (albeit not necessarily finite), and stable under perturbation.

            We can now use continuity of
            \[
                t \mapsto \ind{\mc{L}(tw,\Lambda)},
            \]
            see \cite[Theorem IV-5.17]{Kato95Perturbation}, to conclude that
            \[
                \ind{\mc{L}(w,\Lambda)} = \ind{\mc{L}(\Lambda)} =0
            \]
            for every solution of \eqref{eq:F_zero}. We have shown that $\mc{L}(w,\Lambda)$ is a Fredholm operator of index zero, and the same is then true for $D_w\mc{F}(w,\Lambda)$ by \eqref{eq:L_DF_relationship}, completing the proof.
        \end{proof}

    \subsection{Verification of the compactness property}
        Inspecting \eqref{eq:open_set}, it is clear that a reasonable definition of the increasing sequences $(\mc{Q}^{\lambda}_j)_{j \in \N}$ is to let
        \[
            \mc{Q}^\lambda_j \coloneqq \brac*{(w,\mu) \in \mc{O}_{\lambda} : 1 +\eta \geq \frac{1}{j}, \sgn(\lambda) \hat{\psi}_s|_S \geq \frac{1}{j}, \norm{w}_X + \abs{\mu} \leq j}
        \]
        for each $j \in \N$. These sets are certainly both closed and bounded, and it is evident that, indeed,
        \[
            \mc{O}_\lambda = \bigcup_{j \in \N} \mc{Q}^{\lambda}_j
        \]
        for every $\lambda \neq 0$.

        We will use Schauder estimates to obtain compactness of the intersections
        \begin{equation}
            \label{eq:Qj_intersection}
            \mc{Q}^{\lambda}_j \cap \{(w,\mu) \in \mc{O}_\lambda : \mc{F}(w,\Lambda) = 0\}
        \end{equation}
        for every $\lambda \neq 0$ and $j \in \N$. In order to do so, we will use a \emph{different} way of flattening \eqref{eq:water_wave_problem_stream} than the naive \eqref{def:flattening_transform}, but only in a neighborhood of the surface. This strategy will first give us control of $\eta$, which in turn can be leveraged to control $\hat{\varphi}$.

        \begin{proposition}[Compactness]
            \label{thm:compactness}
            The intersection in \eqref{eq:Qj_intersection} is a compact subset of $X \times \R$ for every $\lambda \neq 0$ and $j \in \N$.
        \end{proposition}
        \begin{proof}
            Without loss of generality, we will assume that $\lambda > 0$, which fixes the sign of $\hat{\psi}_s|_S$. Let $(w,\mu)$ be any point of $\mc{Q}^{\lambda}_j$ such that $(w,\Lambda)$ solves \eqref{eq:F_zero}.
            Pull back $\hat{\psi}$ to $\psi = \hat{\psi} \circ \Pi$ on $\Omega_\eta$ using the naive flattening transform from \eqref{def:flattening_transform}, and observe that
            \begin{equation}
                \label{eq:psi_y_lower_bound_surface}
                \psi_y(x,1+\eta(x)) = \frac{\hat{\psi}_s(x,1)}{1+\eta(x)} \geq \frac{1/j}{1+j} \geq \frac{1}{(1+j)^2}
            \end{equation}
            for all $x \in \R$. Also,
            \begin{equation}
                \label{eq:psi_yy_bound}
                \abs{\psi_{yy}(x,y)} = \frac{\abs{\hat{\psi}_{ss}(x,y/(1+\eta(x)))}}{(1+\eta(x))^2} \leq (1+j)^3
            \end{equation}
            for all $(x,y) \in \Omega_\eta$. Together, \eqref{eq:psi_y_lower_bound_surface} and \eqref{eq:psi_yy_bound} imply the lower bound
            \begin{equation}
                \label{eq:psi_y_lower_bound}
                \psi_y(x,y) \geq \frac{1}{2(1+j)^2}
            \end{equation}
            whenever
            \[
                y \geq 1+ \eta(x) - \frac{1}{2(1+j)^5},
            \]
            through the mean value theorem.

            Furthermore
            \[
                \psi\parn*{x,1+\eta(x)-\frac{1}{2(1+j)^5}} - \mu \leq -\epsilon_j, \quad \epsilon_j \coloneqq \frac{1}{4(1+j)^7},
            \]
            and so we deduce the existence of a streamline $\tilde{\eta} \in X_1$ satisfying both
            \[
                - \frac{1}{2(1+j)^5} \leq \tilde{\eta}(x) - \eta(x) < 0
            \]
            and
            \[
                \psi(x,1+\tilde{\eta}(x)) - \mu = -\epsilon_j
            \]
            for all $x \in \R$.

            \begin{figure}[htb]
                \centering
                \begin{tikzpicture}[use Hobby shortcut]
    %Pre-flattening fade (need to draw first)
    \tikzstyle{fade}=[fill,pattern color=black!50,pattern=north east lines,draw=none,path fading=south]
    
    \draw[thick,dashed,path fading=south] (-5,2.6) -- (-5,2);
    \draw[thick,dashed,path fading=south] (5,2.6) -- (5,2);
    
    \fill[fade] ([in angle=180, Hobby finish, designated Hobby path=next, out angle=0]-5,2.6) ..
    (0,3.4) ..
    ([in angle=180, Hobby finish, designated Hobby path=next, out angle=-90]5,2.6) ..                      
    ([in angle=90, Hobby finish, designated Hobby path=next, out angle=180]5,2) ..
    (0,2) ..
    ([in angle=0, Hobby finish, designated Hobby path=next, out angle=90]-5,2) ..
    (-5,2.6);

    %Pre-flattening
        \filldraw[thick,fill=gray!30,draw=black] ([in angle=180, Hobby finish, designated Hobby path=next, out angle=0]-5,3) ..
                     (0,4) ..
                     ([in angle=180, Hobby finish, designated Hobby path=next, out angle=-90]5,3) ..                      
                     ([in angle=90, Hobby finish, designated Hobby path=next, out angle=180]5,2.6) ..
                     (0,3.4) ..
                     ([in angle=0, Hobby finish, designated Hobby path=next, out angle=90]-5,2.6) ..
                     (-5,3);
        \draw[->,>=Stealth] (2,4) node[right] {$\Omega_\eta \setminus \Omega_{\tilde{\eta}}$} .. (1,4) .. (0,3.7);

    %Arrow
        \draw[thick,->] (0,1.5) -- node[midway, right] {$\Gamma$} (0,0.5);
        
    %Post-flattening fade
    \fill[fade] (-5,-0.5) --
    (5,-0.5) --
    (5,-1) --
    (-5,-1) --
    cycle;
    \draw[thick,dashed,path fading=south] (-5,-0.5) -- (-5,-1);
    \draw[thick,dashed, path fading=south] (5,-0.5) -- (5,-1);
    
    %Post-flattening
    \filldraw[thick,fill=gray!30,draw=black] (-5,0) --
                                             (5,0) --
                                             (5,-0.5) --
                                             (-5,-0.5) --
                                             cycle;
                                             
    \draw[->,>=Stealth] (-2.5,0.25) node[left] {$R_{\epsilon_j}$} .. (-2,0.25) .. (-1,-0.25);

\end{tikzpicture}
                \caption{The setup in the proof of \Cref{thm:compactness}. We ignore whatever is occurring outside of $\Omega_{\eta}\setminus \Omega_{\tilde{\eta}}$.}
                \label{fig:semi_hodograph}
            \end{figure}
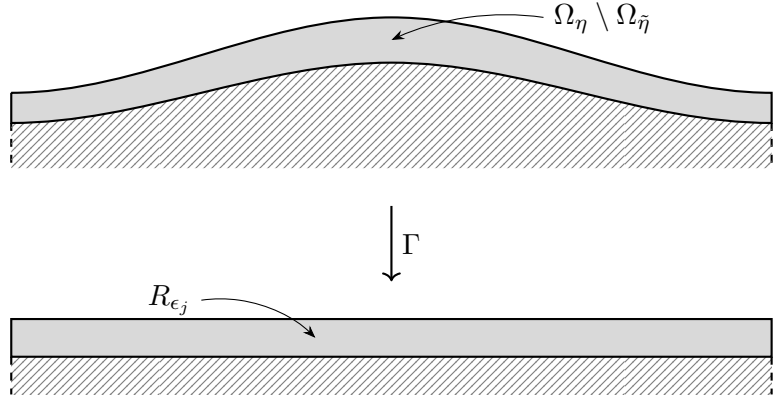

            If we proceed to define the strips
            \[
                R_{\epsilon} \coloneqq \R \times (-\epsilon,0)
            \]
            for $\epsilon > 0$, then the semi-hodograph transform $\map{\Gamma}{\Omega_\eta \setminus \Omega_{\tilde{\eta}}}{R_{\epsilon_j}}$ defined by
            \[
                \Gamma(x,y) = (x,\psi(x,y)-\mu)
            \]
            is a diffeomorphism between the closures of the same sets. The transform has an inverse of the form
            \begin{equation}
                \label{eq:definition_h}
                \Gamma^{-1}(q,p) = (q,h(q,p)),
            \end{equation}
            where the choice of letters for the variables is a matter of convention.

            It is well known that the function $\map{h}{R_{\epsilon_j}}{\R}$ implicitly defined by \eqref{eq:definition_h}, see e.g. \cite{Dubreil-Jacotin34Sur,Constantin04Exact}, satisfies the second order quasi-linear elliptic boundary value problem
            \begin{equation}
                \label{eq:system_for_h}
                \renewcommand\arraystretch{1.33} %Needed to get approximately the right spacing
                \begin{array}{ccl}
                    \mc{S}(h) h = \gamma(p+\mu)h_p^3   &  & \text{in $R_{\epsilon_j}$}, \\
                    1+h_q^2 + (2h-2-\lambda^2)h_p^2= 0 &  & \text{on $p=0$,}
                \end{array}
            \end{equation}
            where we have introduced the differential operator
            \[
                \mc{S}(h) \coloneqq h_p^2\partial_q^2 - 2h_q h_p \partial_{q}\partial_p + (1+h_q^2)\partial_p^2.
            \]
            For the same reasons as in earlier results (such as \cite{Constantin04Exact}), Schauder estimates applied directly to \eqref{eq:system_for_h} do not help us here. However, while \eqref{eq:system_for_h} is not suitable, it follows by straight-forward differentiation that the partial derivative $\theta \coloneqq h_q$ ($h$ is again actually analytic on $\overline{R}_{\epsilon_j}$ due to \Cref{thm:analyticity_of_solutions}) satisfies a similar boundary value problem, namely
            \begin{equation}
                \label{eq:system_for_theta}
                \renewcommand\arraystretch{1.33} %Needed to get approximately the right spacing
                \begin{array}{ccl}
                    \mc{S}(h)\theta = 2h_q(h_{qp}^2 - h_{pp}h_{qq}) + 3\gamma(p+\mu)h_p^2 h_{qp} &  & \text{in $R_{\epsilon_j}$}, \\
                    h_p^3 \theta + h_p h_q \theta_q - (1+h_q^2)\theta_p=0                        &  & \text{on $p=0$,}
                \end{array}
            \end{equation}
            which \emph{can} be used.

            Just like in \eqref{eq:uniform_ellipticity}, we have strict ellipticity in \eqref{eq:system_for_theta}, because
            \[
                h_p^2 \xi_1^2- 2h_ph_q \xi_1\xi_2+ (1+h_q^2)\xi_2^2    \geq \frac{h_p^2}{1+h_p^2 + h_q^2}\abs{\xi}^2
            \]
            for all $\xi=(\xi_1,\xi_2) \in \R^2$. Moreover, it can be shown that (where the exact constant is unimportant)
            \[
                \frac{h_p^2}{1+h_p^2 + h_q^2} \geq \frac{1}{(j+1)^6},
            \]
            whence the strict ellipticity is uniform in the choice of $(w,\mu) \in \mc{Q}_j^\lambda$. The boundary condition at $p=0$ in \eqref{eq:system_for_theta} is also trivially uniformly oblique, in the sense of \cite[(6.76)]{Gilbarg01Elliptic}, because
            \[
                1 + h_q^2 \geq 1,
            \]
            and this is again obviously uniform in the choice of $(w,\mu) \in \mc{Q}^{\lambda}_j$.

            Suppose that $(w_n,\mu_n)_{n \in \N}$, where as usual $w_n = (\eta_n, \hat{\varphi}_n)$, is a sequence in $\mc{Q}^{\lambda}_j$ such that
            \[
                \mc{F}(w_n,\mu_n,\lambda) = 0
            \]
            for all $n \in \N$. Using \eqref{eq:definition_h}, we may define an associated sequence $(h_n)_{n \in \N}$ of functions $\map{h_n}{R_{\epsilon_j}}{\R}$ solving \eqref{eq:system_for_h} with $\mu = \mu_n$. Due to the bounds in the definition of $\mc{Q}^{\lambda}_j$, and the uniform lower bound on $\psi_y$ from \eqref{eq:psi_y_lower_bound}, we infer that this sequence is bounded in $C_{\kappa,\textnormal{e}}^{2,\beta}(\overline{R}_{\epsilon_j})$. As the terms of the corresponding sequence $(\theta_n)_{n \in \N} = (\partial_q h_n)_{n \in \N}$ satisfies \eqref{eq:system_for_theta} with $\mu=\mu_n$ for each $n \in \N$, we deduce from the Schauder estimate in \cite[Theorem 6.30]{Gilbarg01Elliptic} that $(\theta_n)_{n \in \N}$ is bounded in $C_{\kappa,\textnormal{o}}^{2,\beta}(\overline{R}_{\epsilon_j/2})$. Note that, crucially, we do not need a boundary condition at $p =-\epsilon_j$, as we can use \emph{interior} estimates on $R_{\epsilon_j}$ to procure a global estimate on the smaller rectangle.

            It follows that
            \[
                (\partial_x \eta_n)_{n \in \N} = (\theta_n(\wildcard,0))_{n \in \N}
            \]
            is bounded in $C_{\kappa,\textnormal{o}}^{2,\beta}(\R)$, and therefore that the sequence $(\eta_n)_{n \in \N}$ is bounded in $C_{\kappa,\textnormal{e}}^{3,\beta}(\R)$. Recall next that \eqref{eq:water_wave_problem_flattened} is strictly elliptic due to \eqref{eq:uniform_ellipticity}. This ellipticity is again uniform in the choice of $(w,\mu) \in \mc{Q}^{\lambda}_j$, because
            \[
                \frac{1}{(1+\eta)^2 +1 + s^2 \eta_x^2} \geq \frac{1}{(1+j)^2 + 1 + j^2} \geq \frac{1}{2(1+j)^2}
            \]
            on $\overline{\Omega}$ by our definition of $\mc{Q}^{\lambda}_j$.

            Having gained an additional bounded derivative for the surface profile, we can now use the Schauder estimate in \cite[Theorem 6.6]{Gilbarg01Elliptic} on \eqref{eq:water_wave_problem_flattened} to infer that $(\hat{\psi}_n)_{n \in \N}$, and therefore $(\hat{\varphi}_n)_{n \in \N}$, is bounded in $C_{\kappa,\textnormal{e}}^{3,\beta}(\overline{\Omega})$. Finally, by boundedness of $(\mu_n)_{n \in \N}$ and the usual compact embedding of Hölder spaces \cite[Lemma 6.36]{Gilbarg01Elliptic}, we conclude that the sequence $(w_n,\mu_n)_{n \in \N}$ has a convergent subsequence. The intersection \eqref{eq:Qj_intersection} is therefore compact, concluding the proof.
        \end{proof}

\section{Properties of the global curve}

    It is highly desirable to narrow down the alternatives in \Cref{thm:global_bifurcation}. In particular, like in most global bifurcation results, one would typically like to rule out alternative \ref{item:curve_closed} entirely. The reason for this, of course, is that this would guarantee the existence of truly large-amplitude solutions to \eqref{eq:F_zero}. That is not to say that solutions on a hypothetical closed curve are small, but they do not ``blow up'' in the same way that solutions do in the event that \ref{item:curve_good_alternative} occurs.

    Due to the loss of global maximum principles on $\Omega_\eta$, ruling out alternatives in \Cref{thm:global_bifurcation} is significantly more difficult than for non-stagnant waves (and perhaps even impossible in general without making further assumptions on $\gamma$). Alternative \ref{item:curve_closed} was substantially ruled out for the special case of waves with constant vorticity in \cite{Constantin16Global}, albeit in an entirely different framework, but it is not at all clear how to generalize this to more general vorticity distributions.

    Various nodal properties are preserved \emph{near} the surface on the local curve $\mc{K}_{\Lambda^*}^\textnormal{loc}$, but extending these near-surface properties to all of $\mc{K}_{\Lambda^*}$ is challenging. One may imagine an argument akin to the one in the proof of \Cref{thm:compactness}, where one works in a neighborhood of the surface, only this time for the nodal properties. Despite much effort, we have not been able to obtain conclusive results from this, and certainly nothing close to being able to rule out alternative \ref{item:curve_closed}. An additional challenge, is that there are \emph{more} trivial solutions of \eqref{eq:water_wave_problem_stream} than those we have described; namely those with a flat surface $\eta \neq 0$. It could indeed be that the curve loops back to the original bifurcation point, by first passing through one of these trivial solutions.

    For these reasons, and more, we will leave the matter of exploring these alternatives to future work. Still, there are certain things we can quite easily conclude, and which we find worth mentioning. For instance, from \eqref{eq:free_surface_bernoulli} we can immediately infer the upper bound
    \begin{equation}
        \label{eq:eta_upper_bound}
        \eta(t) < \frac{1}{2}(\lambda^*)^2
    \end{equation}
    for every curve parameter $t \in \R$, where the inequality is strict since there are no stagnation points on the surface. Moreover, we note that \emph{if} $\mc{K}_{\Lambda^*}$ were to have a subsequence ending in a wave of greatest height, with a stagnation point at the crest, then it would necessarily have surface deviation $\eta$ precisely equal to the right-hand side of \eqref{eq:eta_upper_bound} at the crest.

    \bibliography{bibliography}

% \bib, bibdiv, biblist are defined by the amsrefs package.
\begin{bibdiv}
\begin{biblist}

\bib{Aasen18Traveling}{article}{
      author={Aasen, A.},
      author={Varholm, K.},
       title={Traveling gravity water waves with critical layers},
        date={2018},
        ISSN={1422-6928},
     journal={J. Math. Fluid Mech.},
      volume={20},
      number={1},
       pages={161\ndash 187},
}

\bib{Buffoni03Analytic}{book}{
      author={Buffoni, B.},
      author={Toland, J.},
       title={{Analytic Theory of Global Bifurcation}},
      series={Princeton Series in Applied Mathematics},
   publisher={Princeton University Press},
     address={Princeton, NJ},
        date={2003},
}

\bib{Constantin11Nonlinear}{book}{
      author={Constantin, A.},
       title={{N}onlinear {W}ater {W}aves with {A}pplications to
  {W}ave-{C}urrent {I}nteractions and {T}sunamis},
   publisher={Society for Industrial and Applied Mathematics},
        date={2011},
      volume={81},
}

\bib{Constantin04Exact}{article}{
      author={Constantin, A.},
      author={Strauss, W.},
       title={Exact steady periodic water waves with vorticity},
        date={2004},
     journal={Comm. Pure Appl. Math.},
      volume={57},
      number={4},
       pages={481\ndash 527},
}

\bib{Constantin16Global}{article}{
      author={Constantin, A.},
      author={Strauss, W.},
      author={V\u{a}rv\u{a}ruc\u{a}, E.},
       title={Global bifurcation of steady gravity water waves with critical
  layers},
        date={2016},
        ISSN={0001-5962},
     journal={Acta Math.},
      volume={217},
      number={2},
       pages={195\ndash 262},
}

\bib{Constantin11Steady}{article}{
      author={Constantin, A.},
      author={V\u{a}rv\u{a}ruc\u{a}, E.},
       title={Steady periodic water waves with constant vorticity: {R}egularity
  and local bifurcation},
        date={2011},
        ISSN={0003-9527},
     journal={Arch. Ration. Mech. Anal.},
      volume={199},
      number={1},
       pages={33\ndash 67},
}

\bib{Crandall71Bifurcation}{article}{
      author={Crandall, M.~G.},
      author={Rabinowitz, P.~H.},
       title={Bifurcation from simple eigenvalues},
        date={1971},
     journal={J. Funct. Anal.},
      volume={8},
       pages={321\ndash 340},
}

\bib{Dancer73Bifurcation}{article}{
      author={Dancer, E.~N.},
       title={Bifurcation theory for analytic operators},
        date={1973},
     journal={Proc. London Math. Soc. (3)},
      volume={26},
       pages={359\ndash 384},
}

\bib{Dubreil-Jacotin34Sur}{thesis}{
      author={Dubreil-Jacotin, M.-L.},
       title={Sur la d{\'e}termination rigoureuse des ondes permanentes
  p{\'e}riodiques d'ampleur finie},
        type={Ph.D. Thesis},
        date={1934},
}

\bib{Ehrnstroem12Steady}{article}{
      author={Ehrnstr{\"o}m, M.},
      author={Escher, J.},
      author={Villari, G.},
       title={Steady water waves with multiple critical layers: {I}nterior
  dynamics},
        date={2012},
        ISSN={1422-6928},
     journal={J. Math. Fluid Mech.},
      volume={14},
      number={3},
       pages={407\ndash 419},
}

\bib{Ehrnstroem11Steady}{article}{
      author={Ehrnstr{\"o}m, M.},
      author={Escher, J.},
      author={Wahl{\'e}n, E.},
       title={Steady water waves with multiple critical layers},
        date={2011},
     journal={SIAM J. Math. Anal.},
      volume={43},
      number={3},
       pages={1436\ndash 1456},
}

\bib{Ehrnstroem08Linear}{article}{
      author={Ehrnstr{\"o}m, M.},
      author={Villari, G.},
       title={Linear water waves with vorticity: {R}otational features and
  particle paths},
        date={2008},
     journal={J. Differential Equations},
      volume={244},
      number={8},
       pages={1888\ndash 1909},
}

\bib{Ehrnstroem15Trimodal}{article}{
      author={Ehrnstr{\"o}m, M.},
      author={Wahl{\'e}n, E.},
       title={Trimodal steady water waves},
        date={2015},
        ISSN={0003-9527},
     journal={Arch. Ration. Mech. Anal.},
      volume={216},
      number={2},
       pages={449\ndash 471},
}

\bib{Gerstner09Theorie}{article}{
      author={Gerstner, F.},
       title={Theorie der {W}ellen},
        date={1809},
     journal={Ann. Phys.},
      volume={32},
       pages={412\ndash 445},
}

\bib{Gilbarg01Elliptic}{book}{
      author={Gilbarg, D.},
      author={Trudinger, N.~S.},
       title={{Elliptic Partial Differential Equations of Second Order}},
      series={Classics in Mathematics},
   publisher={Springer-Verlag, Berlin},
        date={2001},
}

\bib{Groves04Steady}{article}{
      author={Groves, M.~D.},
       title={Steady water waves},
        date={2004},
        ISSN={1402-9251},
     journal={J. Nonlinear Math. Phys.},
      volume={11},
      number={4},
       pages={435\ndash 460},
}

\bib{Kato95Perturbation}{book}{
      author={Kato, T.},
       title={{Perturbation Theory for Linear Operators}},
      series={Classics in Mathematics},
   publisher={Springer-Verlag},
     address={Berlin},
        date={1995},
}

\bib{Kozlov14Dispersion}{article}{
      author={Kozlov, V.},
      author={Kuznetsov, N.},
       title={Dispersion equation for water waves with vorticity and {S}tokes
  waves on flows with counter-currents},
        date={2014},
     journal={Arch. Ration. Mech. Anal.},
      volume={214},
      number={3},
       pages={971\ndash 1018},
}

\bib{Kozlov19Solitary}{article}{
      author={Kozlov, V.},
      author={Kuznetsov, N.~G.},
      author={Lokharu, E.},
       title={Solitary waves on rotational flows with an interior stagnation
  point},
        date={2019},
        note={arXiv:1904.00401},
}

\bib{Kozlov18N}{article}{
      author={Kozlov, V.},
      author={Lokharu, E.},
       title={{$N$}-modal steady water waves with vorticity},
        date={2018},
        ISSN={1422-6928},
     journal={J. Math. Fluid Mech.},
      volume={20},
      number={2},
       pages={853\ndash 867},
}

\bib{Kozlov19Small}{article}{
      author={Kozlov, V.},
      author={Lokharu, E.},
       title={Small-amplitude steady water waves with critical layers:
  {N}on-symmetric waves},
        date={2019},
        ISSN={0022-0396},
     journal={J. Differential Equations},
      volume={267},
      number={7},
       pages={4170\ndash 4191},
}

\bib{Le19Existence}{article}{
      author={Le, Hung},
       title={On the existence and instability of solitary water waves with a
  finite dipole},
        date={2019},
     journal={SIAM Journal on Mathematical Analysis},
      volume={51},
      number={5},
       pages={4074\ndash 4104},
}

\bib{Matioc14Global}{article}{
      author={Matioc, B.-V.},
       title={Global bifurcation for water waves with capillary effects and
  constant vorticity},
        date={2014},
        ISSN={0026-9255},
     journal={Monatsh. Math.},
      volume={174},
      number={3},
       pages={459\ndash 475},
}

\bib{Morrey58Analyticity}{article}{
      author={Morrey, C.~B., Jr.},
       title={On the analyticity of the solutions of analytic non-linear
  elliptic systems of partial differential equations {II}: Analyticity at the
  boundary.},
        date={1958},
        ISSN={0002-9327},
     journal={Amer. J. Math.},
      volume={80},
       pages={219\ndash 237},
}

\bib{Shatah13Travelling}{article}{
      author={Shatah, J.},
      author={Walsh, S.},
      author={Zeng, C.},
       title={Travelling water waves with compactly supported vorticity},
        date={2013},
     journal={Nonlinearity},
      volume={26},
      number={6},
       pages={1529\ndash 1564},
}

\bib{Teschl12Ordinary}{book}{
      author={Teschl, G.},
       title={{Ordinary Differential Equations and Dynamical Systems}},
      series={Graduate Studies in Mathematics},
   publisher={American Mathematical Society, Providence, RI},
        date={2012},
      volume={140},
        ISBN={978-0-8218-8328-0},
}

\bib{Toland96Stokes}{article}{
      author={Toland, J.~F.},
       title={Stokes waves},
        date={1996},
     journal={Topol. Methods Nonlinear Anal.},
      volume={7},
      number={1},
       pages={1\ndash 48},
}

\bib{Varholm16Solitary}{article}{
      author={Varholm, K.},
       title={Solitary gravity-capillary water waves with point vortices},
        date={2016},
     journal={Discrete Contin. Dyn. Syst.},
      volume={36},
      number={7},
       pages={3927\ndash 3959},
}

\bib{Varholm20Stability}{article}{
      author={Varholm, Kristoffer},
      author={Wahlén, Erik},
      author={Walsh, Samuel},
       title={On the stability of solitary water waves with a point vortex},
        date={2020},
     journal={Comm. Pure Appl. Math.},
}

\bib{Wahlen06Steady}{article}{
      author={Wahl\'{e}n, E.},
       title={Steady periodic capillary-gravity waves with vorticity},
        date={2006},
        ISSN={0036-1410},
     journal={SIAM J. Math. Anal.},
      volume={38},
      number={3},
       pages={921\ndash 943},
}

\bib{Wahlen09Steady}{article}{
      author={Wahl{\'e}n, E.},
       title={Steady water waves with a critical layer},
        date={2009},
     journal={J. Differential Equations},
      volume={246},
      number={6},
       pages={2468\ndash 2483},
}

\end{biblist}
\end{bibdiv}
\end{document}